\documentclass[10pt,letter]{amsart}
\usepackage{latexsym}
\usepackage{amscd, amsfonts, eucal, mathrsfs, amsmath, amssymb, amsthm, stmaryrd}
\usepackage{fontenc}
\usepackage{newcent}

\input xy
\xyoption{all}

\usepackage{fullpage}

\newcommand{\Caldararu}{C\u ald\u araru}

\newcommand{\field}[1]{\mathbf #1}
\newcommand{\mf}[1]{\mathfrak #1}
\newcommand{\mc}[1]{\mathcal #1}
\newcommand{\ms}[1]{\mathscr #1}
\newcommand{\widebar}[1]{\overline{#1}}

\newcommand{\R}{\field R}
\renewcommand{\L}{\field L}

\newcommand{\Z}{\field Z}
\newcommand{\Q}{\field Q}

\newcommand{\simto}{\stackrel{\sim}{\to}}
\newcommand{\eps}{\varepsilon}
\renewcommand{\phi}{\varphi}

\renewcommand{\hom}{\operatorname{Hom}}
\newcommand{\shom}{\ms H\!om}
\DeclareMathOperator{\chom}{\underline{Hom}}
\DeclareMathOperator{\uhom}{\underline{Hom}}
\newcommand{\rshom}{\mathbf{R}\shom}
\newcommand{\saut}{\ms A\!ut}

\DeclareMathOperator{\rhom}{\operatorname{{\bf R}Hom}}
\newcommand{\send}{\ms E\!nd}
\newcommand{\rsend}{\mathbf{R}\ms E\!nd}
\DeclareMathOperator{\rend}{\operatorname{{\bf R}End}}
\newcommand{\spec}{\operatorname{Spec}}

\DeclareMathOperator{\Pic}{Pic}
\newcommand{\sPic}{\ms Pic}

\newcommand{\triv}{\textrm{triv}}
\newcommand{\thickslash}{\mathbin{\!\!\pmb{\fatslash}}}
\newcommand{\parf}{\text{\rm parf}}

\newcommand{\m}{\boldsymbol{\mu}}

\newcommand{\G}{\field G} 
\newcommand{\Der}{\operatorname{Der}}

\newcommand{\etale}{\operatorname{\acute{e}t}}
\newcommand{\retale}{\operatorname{r\acute{e}t}}
\newcommand{\ETALE}{\operatorname{\acute{E}T}}

\renewcommand{\H}{\operatorname{H}}
\newcommand{\HH}{\operatorname{\bf H}}

\newcommand{\GL}{\operatorname{GL}}
\newcommand{\PGL}{\operatorname{PGL}}
\DeclareMathOperator{\SL}{\operatorname{SL}}
\DeclareMathOperator{\ext}{\operatorname{Ext}}
\newcommand{\sext}{\ms E\!xt}

\DeclareMathOperator{\cl}{\operatorname{cl}}

\newcommand{\LL}{L}

\DeclareMathOperator{\D}{\operatorname{\bf D}}

\DeclareMathOperator*{\tensor}{\otimes}
\DeclareMathOperator*{\ltensor}{\stackrel{\field L}{\otimes}}

\DeclareMathOperator{\rk}{\operatorname{rk}}

\newcommand{\surj}{\twoheadrightarrow}
\newcommand{\inj}{\hookrightarrow}
\newcommand{\id}{\operatorname{id}}
\DeclareMathOperator{\ann}{\operatorname{Ann}}

\newcommand{\xto}{\xrightarrow}

\DeclareMathOperator{\End}{\operatorname{End}}
\DeclareMathOperator{\aut}{\operatorname{Aut}}
\DeclareMathOperator{\isom}{\operatorname{Isom}}
\DeclareMathOperator{\M}{\operatorname{M}}

\newcommand{\lf}{\textrm{lf}} 

\newcommand{\invlim}{\varprojlim}
\newcommand{\PR}{\ms P\ms R}

\newcommand{\GAz}{\mathbf{GAz}}

\newcommand{\Tw}{\mathbf{Tw}}
\DeclareMathOperator{\mTw}{Tw}

\DeclareMathOperator{\B}{\operatorname{\mathsf B\!}}

\newtheorem{lem}{Lemma}[subsubsection]

\makeatletter
\@addtoreset{lem}{section}
\@addtoreset{lem}{subsection}
\@addtoreset{lem}{subsubsection}
\@addtoreset{lem}{paragraph}
\makeatother

\renewcommand{\thelem}{\ifnum\value{subsubsection}>0{\thesubsubsection.\arabic{lem}}\else{\ifnum\value{subsection}>0{\thesubsection.\arabic{lem}}\else{\thesection.\arabic{lem}}\fi}\fi}

\newtheorem{thm}[lem]{Theorem}
\newtheorem*{theorem}{Irreducibility Theorem}
\newtheorem{prop}[lem]{Proposition}

\newtheorem{cor}[lem]{Corollary}

\theoremstyle{definition}
\newtheorem{defn}[lem]{Definition}
\newtheorem{warning}[lem]{Warning}

\newtheorem{example}[lem]{Example}

\newtheorem{para}[lem]{}

\newtheorem{notn}[lem]{Notation}
\newtheorem{ques}[lem]{Question}

\theoremstyle{remark}
\newtheorem{remark}[lem]{Remark}

\numberwithin{equation}{lem}

\author{Max Lieblich}
\address{Fine Hall, Washington Road, Princeton NJ 08544-1000}
\email{lieblich@math.princeton.edu}
\title{Compactified moduli of projective bundles}
\thanks{The work described in this paper was
  partially supported at various stages by an NSF Graduate Fellowship,
  a Clay Liftoff Fellowship, an NSF Postdoctoral Fellowship, and NSF
  grant DMS-0758391}
\date{}
\begin{document}

\bibliographystyle{plain}

\maketitle

\begin{abstract}
  We present a method for compactifying stacks of $\PGL_n$-torsors
  (Azumaya algebras) on algebraic spaces.  In particular, when the
  ambient space is a smooth projective surface we use our methods to
  show that various moduli spaces are irreducible and carry natural
  virtual fundamental classes.  We also prove a version of the
  Skolem-Noether theorem for certain algebra objects in the derived
  category, which allows us to give an explicit description of the
  boundary points in our compactified moduli problem.
\end{abstract}

\tableofcontents

\section{Introduction}

In this paper, we present a method for constructing compactified
moduli of principal $\PGL_{n}$-bundles on an algebraic space.  As a
demonstration of its usefulness, we will prove the following theorem.

\begin{theorem}[Theorem \ref{T:main}]
  Let $X$ be a smooth projective surface over an algebraically closed
  field $k$ and $n$ a positive integer which is invertible in $k$.
  For any cohomology class $\alpha\in\H^2(X,\m_n)$, the
  stack of stable $\PGL_n$-torsors on $X$ with cohomology class
  $\alpha$ and sufficiently large $c_2$ is of finite type and
  irreducible whenever it is non-empty, and it is non-empty infinitely
  often.
\end{theorem}
\noindent Here the number $c_2$ is meant to be the second Chern class
of the adjoint vector bundle associated to a $\PGL_n$-torsor.  For the
definition of stability of a $\PGL_n$-torsor, we refer the reader to
Definition \ref{D:stab} below; in characteristic $0$ it is equivalent
to slope-stability of the adjoint vector bundle, while in arbitrary
characteristic one quantifies only over ideals in the adjoint (with
respect to 
its natural Azumaya algebra structure).

The Irreducibility Theorem may be viewed as a $0$th order algebraic
version of results of Mrowka and Taubes (see e.g.\ \cite{taubes}) on
the stable topology of the space of $\PGL_n$-bundles; we show that
$\pi_0$ is a singleton.  Our proof arises out of a reduction of the moduli
problem to another recently studied problem: moduli of twisted
sheaves.  Before making a few historical remarks, let us outline the
contents of the paper.

In Section \ref{sec:twist-objects-rigid} we present a general theory
of twisted objects in a stack, including the resulting deformation
theory and the relationship between twisted and untwisted virtual
fundamental classes.  In Section \ref{sec:abstract-compac} we apply
this theory to $\PGL_n$-torsors to show that the stack of twisted
sheaves is naturally a cover of a compactification of the stack of
$\PGL_n$-torsors.  In Section \ref{S:genaz}, we give a
reinterpretation of the results of Section \ref{sec:abstract-compac}
using certain algebra objects of the derived category (generalized
Azumaya algebras), with the ultimate aim being an approach to virtual
fundamental classes for spaces of stable $\PGL_n$-torsors.  The key
result there is Theorem \ref{T:the one}, a version of the
Skolem-Noether theorem for these algebra objects, which we believe
should be of independent interest.

In Section \ref{S:surfaces}, we specialize the whole picture to study
moduli of $\PGL_n$-torsors on smooth projective surfaces.  We develop
the theory of stability in Section \ref{S:stability} and use the known
structure theory of moduli spaces of twisted sheaves on a surface to
prove the Irreducibility Theorem in Section \ref{S:conseq}.  In Sections
\ref{S:p=s on surface} and \ref{S:dtvfc}, we use the interpretation of
the moduli problem in terms of generalized Azumaya algebras to produce
virtual fundamental classes on moduli spaces of stable
$\PGL_n$-torsors on surfaces.  In Section \ref{sec:potent-appl-new},
we record a question due to de Jong regarding potentially new
numerical invariants for division algebras over function fields of
surfaces arising out the virtual fundamental classes constructed in
Section \ref{S:dtvfc}.

\subsection*{Historical remarks}

As has become clear in the history of algebraic geometry, a propitious
choice of compactification of a moduli problem can lead to concrete
results about the original (usually open) subproblem which is being
compactified.  Thus, Deligne and Mumford proved that $\ms M_g$ is
irreducible by embedding it as an open substack of $\widebar{\ms M_g}$
and connecting points by first degenerating them to the boundary.
Similarly, O'Grady approached the moduli of semistable vector bundles
on a surface by considering the larger space of semistable torsion
free sheaves and showing that the boundary admits a stratification by
spaces fibered over moduli of stable vector bundles with smaller
$c_2$.  Combining this inductive structure with delicate numerical
estimates allowed him to prove that the spaces of semistable vector
bundles with sufficiently large $c_2$ are irreducible. (This is very
beautifully explained in Chapter 9 of \cite{h-l}.)  The similarity
between the Irreducibility Theorem above and O'Grady's results for stable
sheaves is traceable to the fact that our compactification is
closely related to the space of twisted sheaves, so that we get a similar inductive structure
on the moduli problem from the geometry of its boundary.

Various attempts have been made at constructing compactified moduli
spaces of $G$-torsors for arbitrary (reductive) groups $G$ (the reader
can consult \cite{gomez-sols2}, \cite{hyeon}, \cite{langer-moduli},
\cite{schmitt2}, and \cite{schmitt} for a sampling of moduli problems
and techniques).  Of course, when $G=\GL_{n}$, one can take torsion
free sheaves, and when $G=\SL_{n}$, one can take torsion free sheaves
with a trivialized determinant.  In the existing literature, most
compactifications proceed (at least in the case where the center of
$G$ is trivial) essentially by encoding a degeneration of a principal
$G$-bundle in a degeneration of its adjoint bundle 
to a torsion-free sheaf along with data which remember the principal
$G$-bundle structure over the open subspace on which the degenerate
sheaf is locally free.

We show how one can analyze the case $G=\PGL_{n}$ using more subtle
methods, which roughly amount to allowing a principal bundle to
degnerate by degenerating its associated adjoint bundle to an object
of the derived category (rather than simply a torsion free sheaf).  By
controlling the nature of these derived objects, we arrive at a
compact moduli stack whose geometry is as tightly controlled as that
of the stack of $\SL_n$-bundles.  This ``tight control'' is formalized
precisely by the covering using twisted sheaves.

\subsection*{Acknowledgments}

The results of this paper are an extension of part of the author's PhD
thesis.  He would like to thank his adviser, A.\ J.\ de Jong, for many
helpful conversations.  The author also greatly benefitted from
contact with Dan Abramovich, Jean-Louis Colliot-Th\'el\`ene, Tom\'as G\'omez, Daniel Huybrechts,
Martin Olsson, and K\=ota Yoshioka during various stages of this
project. He thanks the referee for useful comments.

\section{Notation}

All stacks will be stacks in groupoids.  Thus, given an algebraic
structure (e.g., torsion free sheaf), the stack of objects with that
structure will be assumed to keep track only of isomorphisms.

Given a geometric morphism $f:X\to S$ of topoi and a stack $\ms S$ on $X$,
$f_{\ast}\ms S$ will denote the stack on $X$ whose sections over an
object $T\in S$ are the sections of $\ms S$ over $\pi^{-1}T\in X$.
We will often write $T\to S$ for the map to the final object of $S$ and we
will often use $X\times_S T$ to denote the object $f^{-1}(T)$.  Most
of the topoi we encounter will be the usual \'etale or fppf topos of
an algebraic space or stack, but we do include a few which are
slightly less conventional (e.g., the relative small \'etale topos of
Section \ref{sec:relative-flat-etale}).

A stack over an algebraic space will be called \emph{quasi-proper\/}
if it satisfies the existence part of the valuative criterion of
properness over discrete valuation rings (allowing finite extensions,
as is usually required for algebraic stacks).  Given a Deligne-Mumford
stack $\ms M$ with a coarse moduli space, we will let $\ms
M^{\text{\rm mod}}$ denote the coarse space.  
Given a moduli space (stack) $M$ of sheaves on a proper algebraic
space $X$, we will let $M^{\lf}$ denote the open subspace
parametrizing locally free sheaves.

The notation $(\ms O,\mf m,\kappa)$ will mean that $\ms O$ is a local
ring with maximal ideal $\mf m$ and residue field $\kappa$.

\section{Generalities}\label{S:generalities}

Throughout this section, we fix a geometric morphism of ringed topoi $f:X\to
S$.  (In various subsections, there will be additional hypotheses on the
nature of $X$, $S$, or $f$, but the notation will remain unchanged.)

\subsection{Stacks of sheaves}
\label{sec:stacks-sheaves}

There are various types of sheaves which will be important for us.  We
recall important definitions and set notations in this section.  Let
$Z$ be an algebraic space and $\ms F$ a quasi-coherent sheaf of finite
presentation on $Z$.

The sheaf $\ms F$ is 
\begin{enumerate}
\item \emph{perfect\/} if its image in $\D(\ms
  O_Z)$ is a perfect complex;
\item \emph{pure\/} if for every geometric point
  $z\to Z$ the stalk $\ms F_z$ (which is a module over the local ring $\ms
  O_{z,Z}^{\text{\rm hs}}$) has no embedded primes;
\item \emph{totally supported\/} if the natural
  map $\ms O_Z\to\send(\ms F)$ is injective;
\item \emph{totally pure\/} if it is pure and totally supported.
\end{enumerate}

The key property of perfect sheaves for us will be the fact that one
can form the determinant of any such sheaf.  The reader is referred to
\cite{mumford-knudsen} for the construction and basic facts.

It is clear that all of these properties are local in the \'etale
topology on $Z$ (in the sense that they hold on $Z$ if and only if
they hold on an \'etale cover).  Thus, we can define various stacks on
the small \'etale site of $Z$.  We will write

\begin{enumerate}
\item $\ms T_Z$ for the stack of totally supported sheaves;
\item $\ms T_Z^{\parf}$ for the stack of perfect totally
  supported sheaves;
\item $\ms P_Z$ for the stack of pure sheaves;
\item $\ms P_Z^{\parf}$ for the stack of perfect pure sheaves;
\item if $\ms M$ denotes any of the preceding stacks, we will use $\ms
  M(n)$ to denote the substack parametrizing sheaves with rank $n$ at
  each maximal point of $Z$.
\end{enumerate}
In particular, if $n>0$ then $\ms P^{\parf}_Z(n)$ parametrizes perfect
totally pure sheaves.

\subsection{The relative small \'etale site}
\label{sec:relative-flat-etale}

We recall a few pieces of pure nonsense that will help us apply the
techniques of Section \ref{sec:twist-objects-rigid} below to study moduli
problems.  In this section, we assume that $f$ is a morphism of
algebraic spaces.

\begin{defn}
  The \emph{relative
    small \'etale site of $X/S$\/} is the site whose underlying
  category consists of pairs $(U,T)$ with $T\to S$ a morphism and
  $U\to X\times_S T$ an \'etale morphism.  A morphism $(U,T)\to
  (U',T')$ is an $S$-morphism $T\to T'$ and an $T$-morphism $U\to
  U'\times_T' T$.  A covering is a collection of maps
  $\{(V_i,T)\}\to\{(U,T)\}$ such that $V_i\to U$ form a covering.
\end{defn}
We will denote the topos of sheaves on the relative small \'etale site
by $X_{\retale}$.  There is an obvious geometric morphism of topoi
$X_{\retale}\to T_{\ETALE}$.  (In fact, $X_{\retale}$
is just the ``total space'' of a fibered topos over $T_{\ETALE}$ whose
fiber over $T\to S$ is just the small \'etale topos of $X\times_S T$.)

The relative small \'etale topos is naturally suited to studying
moduli of $T$-flat sheaves on $X$ (as pushforwards of $X$-stacks).

\begin{prop}\label{P:retale}
  Pullback defines a natural equivalence of the category of
  quasi-coherent sheaves on $X$ with the category of quasi-coherent
  sheaves on $X_{\retale}$.  Moreover,
  \begin{enumerate}
  \item there is a stack $\ms C_{X/S}\to X_{\retale}$ whose sections
    over $(U,T)$ parametrize quasi-coherent sheaves on $U$ which are
    $T$-flat and which are locally of finite presentation;
  \item if $\ms M$ denotes any of the stacks from Section
    \ref{sec:stacks-sheaves}, there is a substack $\ms
    M_{X/S}\subset\ms C_{X/S}$ whose sections over $(U,T)$ are
    $T$-flat quasi-coherent sheaves $\ms F$ of finite presentation on
    $U$ such that for each geometric point $t\to T$, the restriction
    $\ms F_t$ lies in $\ms M_{U_t}$;
  \item for each $\ms M$, there is a substack $\ms
    M^{\parf}_{X/S}\subset\ms M_{X/S}$ parametrizing $\ms F$ such that
    each $\ms F_t$ is perfect;
  \item for any $\ms M^{\parf}_{X/S}$ as in the previous item, there
    is a substack $\ms M^{\ms O}_{X/S}(n)\subset\ms M^{\parf}_{X/S}$
    parametrizing sheaves $\ms F$ such that each fiber $\ms F_t$ has
    rank $n$ at each maximal point of $U_t$, along with a global
    trivialization $\det\ms F\simto\ms O_U$.
  \end{enumerate}
\end{prop}
\noindent In the last item, we implicitly use the standard fact that a
sheaf which is perfect on each geometric fiber is perfect.  As an
example, we have that $\ms P^{\ms O}_{X/S}(n)$ denotes the stack on
$X_{\retale}$ whose objects over $(U,T)$ are pairs $(\ms F,\delta)$
with $\ms F$ perfect and $T$-flat, $\delta:\det\ms F\simto\ms O_U$ an
isomorphism, and such that for each geometric point $t\to T$, the
sheaf $\ms F_t$ has rank $n$ at each maximal point of $U_t$.

The proof of Proposition \ref{P:retale} is essentially a sequence of
tautologies and is omitted.  Note that we cannot make any claims about
algebraicity of $\ms C$, $\ms T$, or $\ms P$ because such a statement
is meaningless for stacks on $X_{\retale}$.  However, when $f:X\to S$
is a proper morphism of finite presentation between algebraic spaces,
it is of course standard that the pushforward of $\ms C_{X/S}$ to
$S_{\ETALE}$ is an algebraic stack (and similarly for any $\ms
M_{X/S}$ or $\ms M^{\ms O}_{X/S}$, as the added conditions are open
and the addition of a trivialization of the determinant is 
algebraic).

The following lemma will be useful later.

\begin{lem}\label{L:flat tot supt} 
  Suppose $f:Y\to Z$ is a flat morphism of locally Noetherian schemes
  and $\ms F$ is a $Z$-flat coherent sheaf on $Y$.  If the restriction
  of $\ms F$ to every fiber of $f$ is totally supported, then $\ms F$
  is totally supported on $Y$.
\end{lem}
\begin{proof} 
  We may assume that $X=\spec B$ and $S=\spec A$ are local schemes and
  that $f$ is the map associated to a local homomorphism $\phi:A\to
  B$.  Write $F$ for the stalk of $\ms F$ at the closed point of $B$.
  Choosing generators $x_1,\ldots,x_n$ for $F$, we find a surjection
  $B^n\surj F$ which yields an injection $\End(F)\inj F^n$.  The
  composition of this injection with the natural inclusion of $B$
  sends $1\in B$ to the $n$-tuple $(x_1,\ldots,x_n)\in F^n$.  We will
  show that this map $\iota:B\to F^n$ is an injection.  Note that
  $\iota$ respects base change in the sense that for any $A$-algebra
  $C$, $\iota\tensor_A C$ is the map corresponding to the composition
  $C\to\End_C(F\tensor_A C)\to (F\tensor_A C)^n$.  As the right-hand
  map in that sequence is always an injection, we find that the
  left-hand map is an injection if and only if $\iota\tensor_A C$ is
  an injection.

  We proceed by ``infinitesimal induction'' relative to $A$, i.e.,
  write $A$ with the $\mf m_A$-adic topology as an inverse limit of
  small extensions $\{A_m\}$ with $A_0=k(A)$, the residue field of
  $A$.  We will show that $\invlim\iota_m:\widehat B\to\widehat F^n$
  is an injection.  Krull's theorem and the obvious compatibility then
  show that $\iota$ itself is an injection.

  By hypothesis $\iota_0$ is an injection.  Suppose by induction that
  $\iota_m$ is an injection.  Let $\eps$ generate the kernel of
  $A_{m+1}\to A_m$.  By flatness, there are identifications $\eps
  B_{m+1}\cong(\eps)\tensor_{A_{m+1}} B_{m+1}\cong B_0$ and $\eps
  F_{m+1}^n\cong(\eps)\tensor F_{m+1}^n\cong F_0^n$, and under these
  identifications, $\eps\cdot\iota_{m+1}$ is identified with
  $\iota_0$.  Now consider the diagram
$$\xymatrix{0\ar[r] & \eps B_{m+1}\ar[r]\ar[d] & B_{m+1}\ar[r]\ar[d] & 
  B_m\ar[r]\ar[d] & 0\\
  0\ar[r] & \eps F^n_{m+1}\ar[r] & F^n_{m+1}\ar[r] & F^n_m\ar[r] & 0.
}$$ By the Snake Lemma and the inductive hypothesis, the kernel of the
left-hand vertical map is identified with the kernel of the middle map
(which is $\iota_{m+1}$).  But the left-hand map is identified with
$\iota_0$, hence is injective.
\end{proof}

In particular, a section of $\ms T^{\ms O}_{X/S}$ over $T\to S$ lies
in $\ms T^{\parf}_{X_T}(X_T)$.  This will be essential when we
study relative generalized Azumaya algebras in Section \ref{L:flat tot supt}.

\subsection{Azumaya algebras}\label{S:azumaya}

For the sake of completeness, we recall a few basic facts
about Azumaya algebras, which can be thought of as coherent models for
$\PGL_{n}$-torsors.  We suppose that $f:X\to S$ is a proper morphism
of finite presentation between algebraic spaces.  By abuse of
notation, we will also write $f$ for the induced geometric morphism
$X_{\retale}\to S_{\ETALE}$

Let $G\to S$ a flat linear algebraic $S$-group of finite presentation.
It follows from the definition that $f_{\ast}\B{G_{X}}$ is the stack
of (\'etale) $G$-torsors on $X$, whose sections over an $S$-scheme $T$
are $G_{T}$-torsors on $X_{T}$.

\begin{lem} 
  The stack $f_{\ast}\B{G_{X}}$ is an Artin stack locally of
  finite presentation over $S$.
\end{lem}
\begin{proof}[Sketch of proof] By the usual arguments, we may assume that $S$
  is the spectrum of an excellent Noetherian ring (even a finite type
  $\Z$-algebra if we so desire) and that there is a closed immersion
  $G\inj\GL_{n,S}$ for some $n$.  It is well-known that the stack
  $f_{\ast}\B\GL_{n}$ is an Artin stack locally of finite presentation
  over $S$.  (One can see
  \cite{l-mb} for the case of $X$ projective or apply the main theorem
  of \cite{artin} -- using the standard deformation theory of
  \cite{sga1} and the usual Grothendieck existence theorem of
  \cite{ega3-1} -- in the arbitrary proper case.)  Furthermore,
  extension of structure group yields a 1-morphism
  $\eps:f_{\ast}\B{G}\to f_{\ast}\B{\GL_{n}}$; it suffices to show
  that $\eps$ is representable by algebraic spaces locally of finite
  presentation.  To see this, let $T\to f_{\ast}\B{\GL_{n}}$ be any morphism
  over $S$, corresponding to some $\GL_{n}$-torsor $V$ on $X_T$.  The
  fiber product $f_{\ast}\B{G}\times_{f_{\ast}\B\GL_{n}}T$ is
  identified with the sheaf of reductions of structure group of $V$ to
  $G$, which is simply $V/G$.  Thus, we will be done if we show that
  $f_{\ast}(V/G)$ is an algebraic space locally of finite presentation
  over $T$.

  By Corollary 6.3 of \cite{artin}, the quotient sheaf $V/G$ is
  representable by a separated algebraic space of finite presentation
  over $X_T$.  The fact that $f_{\ast}(V/G)$ is an algebraic space may
  be seen in several ways.  Here is one of them: we can identify it
  with the fiber of $\uhom_T(X_T,V/G)\to\uhom_T(X_T,X_T)$ over the
  section $\id_{X_T}$.  Thus, it suffices to show that
  $\uhom_T(X_T,V/G)$ is an algebraic space locally of finite
  presentation over $T$; this is a standard result, as $X_T$ is proper
  and $V/G$ is separated.  Its algebraicity follows from e.g.\ Artin's
  theorem or from the methods of \cite{lieblich-algebras}.
\end{proof}

In the case of $G=\PGL_{n,S}$, there is a natural closed immersion
$G\inj\GL_{n^{2}}$ given by the action of $\PGL_{n}$ on $\M_{n}(\ms
O)$ by conjugation (the adjoint representation).  In this case, one in
fact has that $\PGL_{n}=\saut_{\text{alg}}(\M_{n}(\ms O))$.  Thus,
there this is a very concrete way to describe those bundles admitting
a reduction of structure group to $\PGL_{n}$: they are those bundles
such that the associated locally free sheaves of rank $n^{2}$ carry
the structure of \emph{Azumaya algebra\/}.

\begin{defn}\label{D:az-alg}
  An \emph{Azumaya algebra\/} $\ms A$ of degree $n$ on a ringed topos
  $T$ is a form of $\M_{n}(\ms O_{T})$.
\end{defn}

More precisely, to give a reduction of structure group on a
$\GL_{n^2}$-torsor is to give a multiplication on the associated
locally free sheaf making it into an Azumaya algebra.  The diagram
$$\xymatrix{1\ar[r] & \G_m\ar[r] & \GL_n\ar[r] & \PGL_n\ar[r] & 1\\
  1\ar[r] & \m_n\ar[r]\ar[u] & \SL_n\ar[r]\ar[u] &
  \PGL_n\ar[r]\ar@{=}[u] & 1}$$ (where the horizontal sequences are
exact in the fppf topology, with the bottom exact in the \'etale
topology only if $n$ is invertible on $S$) gives rise to a diagram of
coboundary maps in non-abelian (flat) cohomology
$$\xymatrix{ & \H^2(T,\G_m) &\\
  \H^1(T,\PGL_n)\ar[ur]\ar[dr] & & \\
  & \H^2(T,\m_n).\ar[uu] & }$$ In Giraud's theory (section V.4.2 of
\cite{giraud}), one can be more precise: given a $\PGL_{n}$-torsor
$P\to T$, the cohomology class $\cl(P)\in\H^{2}(T,\m_{n})$ is precisely
that given by the $\m_{n}$-gerbe of liftings (reductions of structure
group) of $T$ to an $\SL_{n}$-torsor.  In the language of Azumaya
algebras, this is accomplished by looking at the gerbe of
trivializations: a trivialization of $\ms A$ is given by a triple
$(\ms V,\delta,\phi)$ with $\ms V$ a locally free sheaf,
$\delta:\det\ms V\simto\ms O$ a trivialization of the determinant, and
$\phi:\send(\ms V)\simto\ms A$ an isomorphism.

\subsection{Twisted objects and rigidifications}
\label{sec:twist-objects-rigid}

In this section, we give a possible definition for a \emph{twisted
  object\/} in a stack (relative to an abelian gerbe).  We then review
a basic stack-theoretic construction of Abramovich, Corti, and Vistoli
\cite{a-c-v} and show how pushing it forward naturally yields
coverings by stacks of twisted objects.

\subsubsection{Twisted objects}\label{sec:twisted-objects}

Let $\ms S\to X$ be a stack.  Suppose (for the sake of
simplicity) that $A$ is an abelian sheaf on $X$ admitting a central
injection $\chi:A\to\ms I(\ms S)$ into the inertia stack of $\ms S$.
Let $\ms X\to X$ be an $A$-gerbe on $X$.  (Since $A$ is abelian, we
may view this as an $X$-stack along with an identification of $A$ with
the inertia stack $\ms I(\ms X)$.)

\begin{defn}
  An \emph{$\ms X$-twisted section of $\ms S$\/} over $T\to X$ is a
  $1$-morphism $f:\ms X\times_X T\to\ms S$ such that the induced map
  $A\to\ms I(\ms X\times_X T)\to f^{\ast}\ms I(\ms S)$ is identified
  with the pullback under $f$ of the canonical inclusion $\chi:A\to\ms
  I(\ms S)$.
\end{defn}

The collection of $\ms X$-twisted sections of $\ms S$ forms a substack
of the Hom-stack $\chom_X(\ms X,\ms S)$, as the condition on the
inertial morphism is local on the base of any family.  We will write
this substack as $\ms S^{\ms X}$.  Note that there is a natural
central injection $A\to\ms I(\ms S^{\ms X})$ given by acting on a map
$\ms X\to\ms S$ by acting on sections of $\ms S$, or (what amounts to
the same thing by the twisted condition) on the sections
of $\ms X$.

The following transition results will prove useful.

\begin{lem}\label{L:transition-2}
  Let $\ms S$ be an $X$-stack and $\sigma:X\to\ms S$ a section.  There
  is an essentially unique $1$-morphism $\B{\saut(\sigma)}\to\ms S$ sending the
  section corresponding to the trivial torsor to $\sigma$. 
\end{lem}
\begin{proof}
  Let $\widebar{\sigma}\subset\ms S$ be the stack-theoretic image of
  $\sigma$ (so that $\sigma$ factors as an epimorphism
  $X\to\widebar{\sigma}$ followed by a monomorphism
  $\widebar{\sigma}\to\ms S$).  By definition, $\widebar{\sigma}$ is
  the substack of $\ms S$ consisting of objects which are locally
  isomorphic to $\sigma(X)$.  Given an object $Y$ of
  $\widebar{\sigma}$ over some $X$-space $T$, the sheaf
  $\isom_T(Y,\sigma(T))$ is an $\saut(\sigma)_T$-torsor; this defines a
  $1$-morphism $\gamma:\widebar{\sigma}\to\B\saut(\sigma)$.  

  To check that this is a $1$-isomorphism, we choose a cleavage for
  $\ms S$.  It is enough to prove that $\gamma$ is fully faithful on
  fiber categories, as it is clear that any torsor is locally in the
  image of $\gamma$.  Let $Y$ and $Y'$ be two objects of
  $\widebar{\sigma}_T$, and consider the induced map of sheaves
  $\isom(Y,Y')\to\isom(\isom(X,Y),\isom(X,Y'))$.  Since $Y$ and $Y'$
  are both locally isomorphic to $X$, this map of sheaves is trivially
  a surjection.  Thus, we are done once we show that it is injective,
  for which it suffices (by the universality of the argument) to show
  that it is injective on global sections.

  The map described in the statement is simply the $1$-inverse of
  $\widebar{\sigma}\to\B\saut(\sigma)$.
\end{proof}

\begin{prop}\label{P:transition-0}
  Given a section $\sigma:X\to\ms X$, the natural restriction map 
$\ms S^{\ms X}\to\uhom_X(X,\ms S)=\ms S$ is a $1$-isomorphism.
\end{prop}
\begin{proof}
 Given a section of $\ms S$, the injection $A\to\ms I(\ms S)$ combined
 with Lemma \ref{L:transition-2} yields an essentially unique induced map
 $\B{A}\to\ms S$ which respects the $A$-structures on the inertia
 stacks.  This construction gives an isomorphism $\ms S\to \ms
 S^{\B{A}}$.  Using $\sigma$ to identify $\ms X$ with $\B{A}$, we have
 just described the inverse of the natural map given in the statement.
\end{proof}

\begin{prop}\label{P:transition}
  Let $\ms X$ and $\ms Y$ be $A$-gerbes on $X$.  There is a natural
  $1$-isomorphism
$$(\ms S^{\ms X})^{\ms Y}\simto\ms S^{\ms Y\wedge\ms X}$$
of stacks of twisted objects.  
\end{prop}
\begin{proof}
  Consider the diagram
$$\xymatrix{\uhom(\ms Y,\uhom(\ms X,\ms S))\ar@{=}[r] & \uhom(\ms
  Y\times\ms X,\ms S)\\
& \uhom(\ms Y\wedge\ms X,\ms S)\ar[u]\\
(\ms S^{\ms X})^{\ms Y}\ar[uu]\ar@{-->}[r] & \ms S^{\ms Y\wedge\ms X}\ar[u].}$$
The top equality comes from the natural adjunction and the uppermost
vertical right map comes from the natural map $m:\ms Y\times\ms X\to\ms
Y\wedge\ms X$.  The map $A\times A\to\ms I(\ms Y\times\ms X)\to
m^\ast\ms I(\ms Y\wedge\ms X)=A$ is just the addition map, from which
it follows that a unique (up to $2$-isomorphism) dashed arrow exists
filling in the diagram.  The fact that every arrow is either an
equality or an inclusion shows that the dashed arrow is a $1$-isomorphism. 
\end{proof}

\begin{para}\label{Para:cocycley}
  There is a more ad hoc description of $\ms X$-twisted
  objects into terms of a cocycle representing the cohomology class of
  $\ms X$.  This can be useful for its value in
  constructing quick (but perhaps not philosophically satisfying)
  proofs, but we will not use this formalism.  In order to make the
  definition, we must fix a cleavage (pseudo-functor structure) on $\ms S$.  
\end{para}
Given $\ms X$, we can choose a hypercovering $U_{\bullet}\to X$ which
splits $\ms X$, in the following sense:
\begin{enumerate}
\item there is a section $\sigma$ of $\ms X$ over $U_0$, and
\item the two pullbacks of $\sigma$ to $U_1$ are isomorphic, say via
  $\phi$.
\end{enumerate}
Computing the coboundary of $\phi$ and using the fact that $\ms X$ is
an $A$-gerbe yields a $2$-cocycle $a\in\Gamma(U_2,A)$.  It is a
standard fact that this cocycle represents the same cohomology class
as $\ms X$.  Slightly more subtle is the fact that one can explicitly
construct a gerbe from a cocycle on a hypercovering.  (This gerbe is
just the stack of ``twisted $A$-torsors''; we will not describe it in
detail here.)

\begin{defn}
  Given $(U_{\bullet},a)$ as above, a $(U_{\bullet},a)$-twisted
  section of $\ms S$ over $T\to X$ is given by
  \begin{enumerate}
  \item a $1$-morphism $\phi:U_0\times_X T\to\ms S$, and
  \item a $2$-morphism $\psi:(p^1_0)^{\ast}\phi\simto
    (p^1_1)^{\ast}\phi$, where $p^1_0$ and $p^1_1$ are the two natural
    maps $U_1\times_X T\to U_0\times_X T$,
  \end{enumerate}
  subject to the condition that the coboundary
  $\delta\psi\in\aut((p^2_0)^{\ast}\phi)$ is equal to the action of
  $a$ (via the inclusion of $A$ in $\ms I(\ms S)$).
\end{defn}

It is clear that $(U_{\bullet},a)$-twisted objects of $\ms S$ form a
stack on $X$.  

Moreover, given an $\ms X$-twisted object of $\ms S$
over $T$, the construction of $(U_{\bullet},a)$ induces a
$(U_{\bullet},a)$-twisted object of $\ms S$ over $T$.

\begin{prop}\label{P:twisty}
  There is a natural equivalence between the
  stack of $\ms X$-twisted objects of $\ms S$ and
  $(U_{\bullet},a)$-twisted objects of $\ms S$.
\end{prop}
\begin{proof}[Sketch of proof]
  First, let $f:\ms X\to\ms S$ be an $\ms X$-twisted object of $\ms
  S$.  Let $\phi:U_0\to\ms X\to\ms S$ be the composition of $f$ with
  the map coming from the chosen trivialization of $\ms X$ over $U_0$.
  Via the cleavage on $\ms S$, the two maps $U_1\to U_0$ give an
  isomorphism  $\psi:(p^1_0)^{\ast}\phi\simto
    (p^1_1)^{\ast}\phi$ of the pullbacks.  The condition that $f$ be
    $\ms X$-twisted shows that the action of the coboundary is
    precisely multiplication by $a$, giving a $(U_\bullet,a)$-twisted
    object of $\ms S$.

    By descent theory, the statement that this gives an equivalence
    boils down to the proposition that a morphism $\ms X\to\ms S$ is
    equivalent to a natural transformation between fibered categories.
    (This requires some careful justification, which for example comes
    from the realization of $\ms X$ as the stack of $a$-twisted
    torsors with respect to the pair $(U_\bullet,a)$.  Since we will
    not use this formalism in this paper, we will not go into the
    rather unpleasant details.)
\end{proof}

\subsubsection{Pushing forward
  rigidifications}\label{sec:push-forw-rigid}

Let $\ms S\to X$ be a stack on $X$ with inertia stack $\ms I(\ms
S)\to\ms S$.  Suppose $A$ is an abelian sheaf on $X$ admitting a
central injection $A_{\ms S}\inj\ms I(\ms S)$.  In \cite{a-c-v},
Abramovich, Corti, and Vistoli construct the \emph{rigidification of
  $\ms S$ along $A$\/}, denoted $\ms S\thickslash A$ (using the
notation of section 5 of \cite{romagny}).  It is characterized by a
universal property: there is a 1-morphism $\ms S\to\ms S\thickslash A$
which is 1-universal among morphisms $\phi:\ms S\to\ms T$ for which
$A_{\ms S}$ is in the kernel of the induced map $\ms I(\ms
S)\to\phi^{\ast}\ms I(\ms T)$.  We will freely use the standard fact that
$\ms S\to\ms S\thickslash A$ is representable by $A$-gerbes.

\begin{remark}
  While all existing references discuss
  rigidifications only for algebraic stacks on the category of
  $S$-schemes for some scheme $S$, the abstract nonsense works
  perfectly well for stacks on any site.  We will implicitly use this
  in what follows.
\end{remark}

In this section we study the morphism $f_{\ast}\ms S\to f_{\ast}(\ms
S\thickslash A)$.  Given an $S$-space $T\to S$ and a $1$-morphism
$\gamma:T\to f_{\ast}(\ms S\thickslash A)$, there results an
$A_T$-gerbe on $X\times_S T$, coming from the fact that $\ms S\to\ms
S\thickslash A$ is represented by $A$-gerbes and the fact that $T\to
f_{\ast}(\ms S\thickslash A)$ corresponds to a morphism $X\times_S
T\to\ms S\thickslash A$.

\begin{defn}\label{def:gerbe}
  With the above notation, the $A$-gerbe associated to $\gamma$ will
  be denoted $\ms X_\gamma$ and called the \emph{gerbe of $\gamma$}.
  The class of $\ms X_\gamma$ in $\H^2(X\times_S T,A)$ will be called
  the \emph{(cohomology) class of $\gamma$}.
\end{defn}

\begin{prop}
  Let $\ms X\to X$ be an $A$-gerbe.  There is a canonical isomorphism
  $\ms S^{\ms X}\thickslash A\cong\ms S\thickslash A$.  Moreover, for
  any $T\to S$, a $1$-morphism $\gamma:T\to f_{\ast}(\ms S\thickslash
  A)$ lifts to a $1$-morphism $T\to f_{\ast}(\ms S^{\ms X_{\gamma}})$.
\end{prop}
\begin{proof}
  Consider the following diagram
\begin{equation}\label{E:diag}
\xymatrix{
\ms S^{\ms X}\ar[r]\ar@{-->}[d] & \uhom_X(\ms X,\ms S)\ar[d]\\
\uhom_X(X,\ms S\thickslash A)\ar@{=}[d]\ar[r] & \uhom_X(\ms X,\ms S\thickslash
A)\\
\ms S\thickslash A.}
\end{equation}
The $\ms X$-twisted condition shows that the dashed arrow exists, and
since every arrow in question is a monomorphism, the dashed arrow is
unique up to $2$-isomorphism.  By the universal property of
$\thickslash$, there results a natural morphism $\nu:\ms S^{\ms
  X}\thickslash A\to\ms S\thickslash A$.

To show that $\nu$ is an equivalence, we may work locally on $X$ and
assume that $\ms X$ is trivial.  In this case, Proposition \ref{P:transition-0}
shows that the dashed arrow in (\ref{E:diag}) is the image of a
$1$-isomorphism 
$\ms S\simto\ms S^{\ms X}$ which respects the $A$-structures.  It
follows that the map on rigidifications is an isomorphism, as desired.
\end{proof}

\begin{defn}
  Given an $A$-gerbe $\ms X\to X$, let $f^{\ms X}_{\ast}(\ms
  S\thickslash A)$ denote the stack-theoretic image of $f_{\ast}(\ms
  S^{\ms X})$ under the natural map $f_{\ast}(\ms S^{\ms X})\to
  f_{\ast}(\ms S\thickslash A)$.  We will call $f^{\ms X}_{\ast}(\ms
  S\thickslash A)$ the \emph{$\ms X$-twisted part of $f_{\ast}(\ms
    S\thickslash A)$\/}.
\end{defn}

\begin{lem}\label{L:pushforwardcovering}
  Given an $A$-gerbe $\ms X\to X$ and a $1$-morphism $\phi:T\to f^{\ms
    X}_{\ast}(\ms S\thickslash A)$, there is an \'etale surjection $U\to
  T$ and an isomorphism $T\times_{f^{\ms X}_{\ast}(\ms A\thickslash
    A)}f_{\ast}(\ms S^{\ms X})|_U\cong f_{\ast}\B{A}|_U$.
\end{lem}
\begin{proof}
  By construction $f_{\ast}(\ms S^{\ms X})\to f^{\ms X}_{\ast}(\ms
  S\thickslash A)$ is an epimorphism of stacks, so there is some $U\to
  T$ such that $\phi|_U$ lifts into $f_{\ast}(\ms S^{\ms X})$.  Thus,
  it suffices to show that if $\phi$ lifts to $f_{\ast}(\ms S^{\ms
    X})$ then the fiber product is isomorphic to $f_{\ast}\B{A}$.  In
  this case the gerbe $\ms X_{\gamma}\to X\times_S T$ is isomorphic to
  $\B{A}$.  The result follows from the compatibility of the formation
  of fiber product with pushforward.
\end{proof}

\subsubsection{Deformation theory}
\label{sec:defmn-thry}

In this section we assume that $f$ is a proper morphism of finite
presentation between algebraic spaces and $A$ a tame constructible
abelian \'etale sheaf.

\emph{We assume throughout this section that $f_{\ast}A$ and
  $\R^1f_{\ast}A$ are finite \'etale over $S$\/}.  (It is known that
they are both constructible; if $f$ is smooth and $A$ is the pullback
of a finite \'etale group scheme then this hypothesis will be
satisfied.  This will be the case in applications of interest to us.)

\begin{lem}\label{L:BA-structure}
  There is a natural morphism $f_{\ast}\B{A}\to\R^1f_{\ast}A$ which
  realizes $f_{\ast}\B{A}$ as a $f_{\ast}A$-gerbe over
  $\R^1f_{\ast}A$.
\end{lem}
\begin{proof}
  Recall that $\R^1f_{\ast}A$ is defined as the
  sheafification of the functor $(T\to S)\mapsto\H^1(X\times_S T,A)$
  on the big \'etale site of $S$.  A section of $f_{\ast}\B{A}$ over
  $T\to S$ corresponds to an $A$-torsor $\ms T$ on $X\times_S T$.  In
  fact, the set of isomorphism classes of objects of $f_{\ast}\B{A}$
  over $T$ is naturally isomorphic to $\H^1(X\times_S T,A)$.  Thus,
  $\R^1f_{\ast}A$ is the sheafification of the stack $f_{\ast}\B{A}$,
  from which it immediately follows that
  $f_{\ast}\B{A}\to\R^1f_{\ast}A$ is a gerbe.  It remains to identify
  the inertia stack $\ms I(f_{\ast}\B{A})$.  But the inertia stack of
  $\B{A}$ is naturally identified with $A$, from which it follows that
  there is a natural isomorphism $f_{\ast}A\simto\ms
  I(f_{\ast}\B{A})$.
\end{proof}

\begin{cor}\label{C:no-cotgt-cplx}
  The cotangent complex of $f_{\ast}\B{A}$ over $S$ is trivial.
\end{cor}
\begin{proof}
  By the usual triangles and the fact that $\R^1f_{\ast}A$ is \'etale
  over $S$, it suffices to show that if $\Gamma$ is a finite \'etale
  group scheme and $\ms Y\to Y$ is a $\Gamma$-gerbe then the cotangent
  complex of $\ms Y$ over $Y$ is trivial.  But this follows
  immediately from the fact that $\ms Y$ there is a surjection
  $U\to\ms Y$ such that $U$ is \'etale over both $\ms Y$ and $Y$.
\end{proof}

\begin{cor}
  The natural map $\chi:f_{\ast}(\ms S^{\ms X})\thickslash
  f_{\ast}A\to f_{\ast}^{\ms X}(\ms S\thickslash A)$ is representable
  by finite \'etale covers.
\end{cor}
\begin{proof}
  By Lemma \ref{L:pushforwardcovering}, the fiber of $\chi$ is locally
  $f_{\ast}\B{A}\thickslash f_{\ast}A$.  Applying Lemma \ref{L:BA-structure}
  shows that this is precisely $\R^1f_{\ast}A$.
\end{proof}

\begin{lem}\label{P:alg alg stack stack}  If $f:\ms S\to\ms S'$ is 
  a map of $S$-stacks which is representable by fppf morphisms of
  algebraic stacks then $\ms S$ is algebraic if and only if $\ms S'$
  is.
\end{lem}
\begin{proof} First, we show that the diagonal of $\ms S$ is
  separated, quasi-compact, and representable by algebraic spaces if
  and only if the same is true for $\ms S'$.  To this end, let
  $T'\to\ms S'\times\ms S'$ be a morphism with $T'$ an affine scheme.
  Consider the diagram

$$\xymatrix{& \ms S \ar[rr]\ar'[d][dd] & & \ms S\times\ms S\ar[dd] \\
  I \ar[ur]\ar[rr]\ar[dd] & & T \ar[ur]\ar[dd] \\
  & \ms S' \ar'[r][rr] & & \ms S'\times\ms S' \\
  I'\ar[rr]\ar[ur] & & T' \ar[ur] }$$ whose terms we now explain.  The
sheaf $I'$ is the pullback of $T'$ along the diagonal.  By assumption,
the fiber product $\ms S\times\ms S\times_{\ms S'\times \ms S'}T'$ is
an algebraic stack over $T'$ with fppf structure morphism.  Thus, we
may let $T$ be a scheme which gives a smooth cover, and then we let
$I$ be the pullback sheaf of $T$ along the diagonal of $\ms S$.  We
see that $I\to I'$ is relatively representable by fppf morphisms of
algebraic spaces.  By a result of Artin \cite[10.1]{l-mb}, $I$ is an
algebraic space if and only if $I'$ is.

It remains to show that $\ms S$ has a smooth cover by an algebraic
space if and only if $\ms S'$ does.  In fact, it suffices to replace
the word ``smooth'' by ``fppf,'' by Artin's theorem
[\textit{ibid\/}.].  But then the statement is clear.
\end{proof}

\begin{prop}\label{P:stacky-means-stacky}
  Given an $X$-stack $\ms S$ and an $A$-gerbe $\ms X\to X$, the stack
  $f_{\ast}(\ms S^{\ms X})$ is an Artin (resp.\ DM) stack if and only
  if the stack $f_{\ast}^{\ms X}(\ms S\thickslash A)$ is an Artin
  (resp.\ DM) stack.
\end{prop}
\begin{proof}
  This follows immediately from Lemma \ref{L:pushforwardcovering},
  Lemma \ref{L:BA-structure}, and Lemma \ref{P:alg alg stack stack}.
\end{proof}

\begin{para}
  We apply the above considerations to give a relation between certain
  virtual fundamental classes on $f_{\ast}(\ms S^{\ms X})$ and
  $f_{\ast}^{\ms X}(\ms S\thickslash A)$.
\end{para}
\begin{prop}\label{P:virt-fund}
  Let $\xi:\ms Z\to\ms W$ be a map of $S$-stacks.  Suppose there is a
  tame finite \'etale group scheme $G\to S$ and a central injection
  $G_{\ms Z}\inj\ms I(\ms Z)$ such that
  \begin{enumerate}
  \item the map $G_{\ms Z}\to\ms I(\ms Z)\to\xi^{\ast}\ms I(\ms W)$ is
    trivial;
  \item the induced $1$-morphism $\ms Z\thickslash G\to\ms W$ is
    representable by finite \'etale morphisms of degree invertible on
    $\ms W$.
  \end{enumerate}
  Given a perfect complex $\ms E\in\D(\ms W)$ and a map $\xi^{\ast}\ms
  E\to\LL_{\ms Z/S}$ which gives a perfect obstruction theory, there
  is a map $\ms E\to\LL_{\ms W/S}$ giving a perfect obstruction
  theory.
\end{prop}
\begin{proof}
  Write $\widebar{\ms Z}:=\ms Z\thickslash G$.  We have a diagram $\ms
  Z\to\widebar{\ms Z}\to\ms W$ with the property that the relative
  cotangent complex of any pair vanishes.  We first claim that any map
  $\xi^{\ast}\ms E\to\LL_{\ms Z/S}$ is the pullback of a map $\ms
  E|_{\widebar{\ms Z}}\to\LL_{\widebar{\ms Z}/S}$.  This follows from
  the fact that $A$ acts trivially on the sheaves making up the
  complexes $\ms E$ and $\LL_{\ms Z/S}$ and the usual description of
  sheaves on gerbes in terms of the representation theory of $A$.
  
  Thus, to prove the result, we are reduced to the case where $\ms
  Z\to\ms W$ is representable by finite \'etale maps with invertible
  degrees.  In this case, there is a splitting trace map given by
  dividing the trace of the covering by the degree.  Note that
  $\LL_{\ms Z/S}=\xi^{\ast}\LL_{\ms W/S}$, so that the perfect
  obstruction theory becomes a map $\alpha:\xi^{\ast}\ms
  E\to\xi^{\ast}\LL_{\ms W/S}$.  Taking the splitting trace produces a
  map $\ms E\to\LL_{\ms W/S}$, which is a perfect obstruction
  theory because it is a summand of $\xi_{\ast}\alpha$.
\end{proof}

\begin{cor}\label{C:v-c}
  A complex in $\D(f_{\ast}^{\ms X}(\ms S\thickslash A))$ can be
  realized as a perfect obstruction theory if and only if its pullback
  to $f_{\ast}(\ms S^{\ms X})$ can be realized as a perfect
  obstruction theory.
\end{cor}

\section{Compactified moduli of $\PGL_n$-torsors: an abstract
  approach}
\label{sec:abstract-compac}

In this section we compactify the moduli of $\PGL_n$-torsors using the
techniques of Section \ref{sec:twist-objects-rigid} and use the
structure of our compactification to prove the Irreducibility Theorem.
In Section \ref{S:genaz} we will give a more concrete description of
the abstract compactification we construct here and use it to describe
the virtual fundamental class on the moduli stack.

Throughout this section, $f:X\to S$ will be a proper flat morphism of
finite presentation between algebraic spaces, $n$ will be an integer
invertible on $S$, and $\pi:\ms X\to X$ will be a fixed $\m_n$-gerbe.  We
will abuse notation and let $f$ also stand for the geometric morphism
$X_{\retale}\to S_{\ETALE}$.

\subsection{Twisted sheaves}
\label{sec:tw-sh}

We briefly describe how the theory
developed in Section \ref{sec:twisted-objects} works out in the case
of twisted objects of the stack of coherent sheaves on $X/S$.  We fix
the $\m_n$-gerbe $\ms X\to X$ and do the twisting with respect to the
natural inclusion of $\m_n$ into the inertia stack of $f_{\ast}\ms
C_{X/S}$ (see Proposition \ref{P:retale}).

\begin{defn}
  An $\ms X$-twisted object of $f_{\ast}\ms C_{X/S}$ over $T\to S$ is
  called \emph{a flat family of $\ms X$-twisted coherent sheaves
    parametrized by $T$\/}.
\end{defn}

If the fibers of the family are torsion free, we will speak of a flat
family of torsion free $\ms X$-twisted sheaves, etc.  The reader is
referred to paragraph 2.2.6.3 of \cite{twisted-moduli} for a
discussion of associated points, purity, and torsion free sheaves on
Artin stacks.

Concretely, an $\ms X$-twisted sheaf is a sheaf $\ms F$ on $\ms X$
such that the representation of $\m_n$ on each geometric fiber of $\ms
F$ is given by scalar multiplication.  These sheaves were originally
introduced by Giraud in \cite{giraud} and have found various recent
applications in mathematical physics and in algebra.

\begin{notn}
  The stack of totally pure $\ms X$-twisted coherent sheaves with rank
  $n$ and trivialized determinant will be denoted $\Tw_{\ms X/S}(n,\ms
  O)$.
\end{notn}

It is relatively straightforward to prove that $\Tw_{\ms X/S}(n,\ms
O)$ is an Artin stack locally of finite presentation over $S$.  This
is done in detail in section 2.3 of \cite{twisted-moduli}.  Earlier
work on twisted sheaves in the context of elliptic fibrations and $K3$
surfaces was carried out by \Caldararu\ \cite{caldararu,caldararu2},
and a study of their moduli for projective varieties, with a
description of the moduli spaces associated to $K3$ and Abelian
surfaces, by Yoshioka \cite{yoshioka}.  Applications of Yoshioka's
results to a conjecture of \Caldararu\ were discovered by Huybrechts
and Stellari (described in the appendix to \cite{yoshioka}).  The
abstract approach taken here has also proven useful in the study of
certain arithmetic questions \cite{period-index-paper}.  When $\ms X$
admits a locally free twisted sheaf $\ms V$ (i.e., when its class in
$\H^2(X,\G_m)$ lies in the Brauer group of $X$), $\ms X$-twisted
sheaves are equivalent to $\pi_\ast\send(\ms V)$-modules, where the
problem of constructing moduli under various stability conditions was
first studied by Simpson \cite{simpson}.  The case in which the
$\pi_\ast\send(\ms V)$-modules have rank $1$ was studied by Hoffmann
and Stuhler \cite{h-s}; they also produced a symplectic structure on
the moduli space when $X$ is a $K3$ or abelian surface, giving results
analogous to those of Yoshioka.

\subsection{Compactification by rigidification}
\label{sec:compac-by-rig}

The natural map $\SL_n\to\PGL_n$ gives an ``extension of structure
group'' morphism $\eps:\B\SL_n\to\B\PGL_n$.

\begin{lem}
  The map $\eps$ induces an isomorphism
  $\B\SL_n\thickslash\m_n\simto\B\PGL_n$.
\end{lem}
\begin{proof}
  Given a stack $\ms S$, there is a natural equivalence of categories
  between morphisms $\B\SL_n\to\ms S$ and $\SL_n$-equivariant objects
  of $\ms S$.  (The reader is referred to section 3.A of
  \cite{kovacs-2006} for a description of this equivalence.)  On the
  other hand, there is clearly a natural equivalence between
  $\PGL_n$-equivariant objects of $\ms S$ and $\SL_n$-equivariant
  objects on which the $\m_n\subset\SL_n$ acts trivially.  But these
  correspond precisely to morphisms $\B\SL_n\to\ms S$ such that the
  induced map on inertia annihilates $\m_n\subset\ms
  I(\B\SL_n)$.  The lemma follows from the universal property of the
  rigidification.
\end{proof}

Taking the associated locally free sheaf with trivialized determinant
yields an inclusion $\B\SL_n\subset\ms T^{\ms O}_{X/S}(n)$.  (Note
that the natural target is not $\ms P^{\ms O}_{X/S}(n)$ unless the
fibers of $X/S$ are Cohen-Macaulay.)  Moreover, there is a natural
inclusion $\m_n\inj\ms I(\ms T^{\ms O}_{X/S}(n))$ extending the
inclusion over $\B\SL_n$.  It follows that there is an inclusion
$\B\PGL_n\inj(\ms T^{\ms O}_{X/S}(n)\thickslash\m_n)$.

There is a natural morphism from
$\chi:f_{\ast}\B\PGL_n\to\R^2f_{\ast}\m_n$ which we may define as
follows.  (Note that since $X$ is proper over $S$, the sheaf
$\R^2f_{\ast}\m_n$ on $S_{\ETALE}$ is a quasi-finite algebraic space
of finite presentation by Artin's theorem.  A proof in terms of
algebraic spaces may be found in the last chapter of
\cite{repr-artin}.)  Given an object of
$f_{\ast}\B\PGL_n=f_{\ast}(\B\SL_n\thickslash\m_n)$ over some $T\to
S$, there is an associated $\m_n$-gerbe on $X\times_S T$ (see
Definition \ref{def:gerbe}), and we simply take the image in
$\H^0(T,\R^2f_{\ast}\m_n)$.

\begin{lem}
  The stack $f_{\ast}(\ms T^{\ms O}_{X/S}(n)\thickslash\m_n)$ is an
  Artin stack locally of finite presentation over $S$.  If in addition
  $f$ is smooth then the stack is quasi-proper.
\end{lem}
\begin{proof}
  Since $\R^2f_{\ast}\m_n$ is an algebraic space, it suffices to show
  that $\chi$ makes $f_{\ast}(\ms T^{\ms O}_{X/S}(n))$ into an
  algebraic $(\R^2f_{\ast}\m_n)$-stack of finite presentation.  To
  prove this, it suffices to work locally on $\R^2f_{\ast}\m_n$.
  Thus, as any section of $\R^2f_{\ast}\m_n$ (and in particular, the
  ``universal section'' given by the identity map) is locally
  associated to the cohomology class of a $\m_n$-gerbe, we see that it
  suffices to prove that, given a $\m_n$-gerbe $\ms X\to X$, the stack
  $f_{\ast}^{\ms X}(\ms T^{\ms O}_{X/T}(n)\thickslash\m_n)$ is an
  Artin stack locally of finite presentation.  Applying
  Proposition \ref{P:stacky-means-stacky}, we see that it suffices to prove that
  $f_{\ast}((\ms T^{\ms O}_{X/T}(n))^{\ms X})$ is an Artin stack
  locally of finite presentation.  But this is an open substack of the
  stack of perfect $\ms X$-twisted sheaves with trivialized
  determinant, which is known to be Artin and locally of finite
  presentation.  (For the proof that it is an Artin stack, the reader
  is referred to section 2.3 of \cite{twisted-moduli}.  The condition
  that the fibers be perfect is clearly an open condition, and closed
  if $X/S$ is smooth.)

  Suppose $f$ is smooth, and let $(R,(t),\kappa)$ be a discrete
  valuation ring over $S$ with fraction field $K$.  Suppose $\spec K\to
  f_{\ast}(\ms T^{\ms O}_{X/S}(n)\thickslash\m_n)$ is a $1$-morphism.
  We may suppose without loss of generality that $S=\spec R$.  Since
  $f$ is proper and smooth, $\R^2f_{\ast}\m_n$ is finite and \'etale
  over $R$.  Making a finite base change (which is permitted in the
  stacky version of the valuative criterion), we may assume that
  $\R^2f_{\ast}\m_n$ is a disjoint union of sections over $\spec R$.
  It follows that to prove that $f_{\ast}(\ms T^{\ms
    O}(n)\thickslash\m_n)$ is quasi-proper, it suffices to prove that
  $f_{\ast}((\ms T^{\ms O}_{X/S}(n))^{\ms X})$ is quasi-proper, where
  $\ms X\to X$ is an arbitrary $\m_n$-gerbe.  Since $X$ is regular,
  the condition that the twisted sheaf be perfect is trivial, and the
  result comes down to showing that given a discrete valuation ring
  $R$ and a torsion free $\ms X$-twisted sheaf $\ms F$ of rank $n$
  with trivialized determinant over the generic point of $R$, there is
  an extension of $\ms F$ to a flat family over a finite flat
  extension of $R$ such that the trivialization of the determinant
  extends.

  Let $K$ be the fraction field of $R$ and $\kappa$ its residue field.
  It is easy to see that any flat extension $\ms G$ of $\ms F$ will
  have trivial determinant (as all invertible sheaves on $\spec R$ are
  trivial).  Choose an isomorphism $\iota:\det\ms G\simto\ms O$.
  Composing with the fixed generic isomorphism $\det\ms F\simto\ms
  O_{X_K}$ yields an injection $\alpha:\det\ms G\to\ms O_{X_K}$ (the
  latter being viewed as a sheaf on $X$ by pushforward from $X_K$).
  Since $X$ is geometrically connected, the trivial invertible $\ms
  O_{X_R}$-subsheaves of $\ms O_{X_K}$ all have the form $t^s\ms
  O_{X_R}$ for some $s\in\Z$.  Taking an $n$th root of $t$ if
  necessary (which may result in a finite extension of $R$), we may
  assume that $s=ns'$ for some integer $s'$.  Replacing $\ms G$ by
  $\ms G(t^{-s'})$ yields $\det\ms G(t^{-s'})=(\det\ms G)(t^{-s})$.
  Thus, via $\iota$ and the given isomorphism $\det\ms F\simto\ms O$,
  $\det\ms G(t^{-s'})$ gets identified with $t^{-s}t^s\ms O$, i.e.,
  $\iota$ yields a trivialization of $\det\ms G(t^{-s'})$ which
  extends that of $\det\ms F$, as desired.
\end{proof}

\begin{lem}
  The natural map $f_{\ast}\B\PGL_n\to f_{\ast}(\ms T^{\ms
    O}_{X/S}(n)\thickslash\m_n)$ is representable by open immersions.
\end{lem}
\begin{proof}
  It again suffices to prove this for $f_{\ast}\B\SL_n^{\ms X}$ and $f_{\ast}(\ms
  T^{\ms O}_{X/S}(n))^{\ms X}$, where we note that $f_{\ast}\B\SL_n^{\ms X}$
  parametrizes locally free $\ms X$-twisted sheaves of rank $n$
  and trivialized determinant and hence constitutes an open substack,
  as desired.
\end{proof}

\begin{para}\label{P:pure-mod}
  When the fibers of $X/S$ are Cohen-Macaulay, the entire discussion
  from the beginning of the section until the present paragraph also
  yields a compactification coming from the induced inclusion
  $\B\PGL_n\inj\ms P^{\ms O}_{X/S}$.  We omit the details; the
  statements of the results are literally identical, with $\ms P$
  replacing $\ms T$.  Since $\ms T^{\ms O}_{X/S}$ is much larger than
  $\ms P^{\ms O}_{X/S}$, it is preferable to use the latter whenever
  possible.  Thus, for example, if $X/S$ is a smooth morphism, then
  there results an open immersion into a quasi-proper Artin stack
  $f_{\ast}\B\PGL_n\inj f_{\ast}(\ms P^{\ms
    O}_{X/S}(n)\thickslash\m_n)$.  This latter stack will play an
  important role in what follows.  We endow it with the following
  notation.
\end{para}

\begin{notn}
  Given a $\m_n$-gerbe $\ms X$, let $\ms M^{\ms X}_n:=f_{\ast}^{\ms
    X}(\ms P^{\ms O}_{X/S}\thickslash\m_n)$.
\end{notn}
There is a surjective map $\Tw_{\ms X/S}(n,\ms O)\to\ms M^{\ms X}_n$
which is universally closed and submersive.

\section{An explicit description of $\ms M_n^{\ms X}$: generalized
  Azumaya algebras}\label{S:genaz}
In this section, we use certain algebra objects of the derived
category to give a concrete description of $\ms M_n^{\ms X}$.  Using
this description, we will show that when $X$ is a smooth projective
surface and $\ms X\to X$ has order $n$ in $\H^2(X,\G_m)$ then $\ms
M_n^{\ms X}$ has a virtual fundamental class.

\subsection{Derived Skolem-Noether}\label{S:skolem-noether}

In this section, we work primarily in the derived category of modules
over a local commutative ring $(\ms O,\mf m, k)$.  For the sake of a
smoother exposition, we assume that $\ms O$ is Noetherian, but this is
unnecessary.  On occasion, we will work in the category of chain
complexes.  However, we will use the word ``complex'' in both
settings; it will be clear in context whether we
mean an object of $\D(\ms O)$ or an object of $\operatorname{K}(\ms
O)$.  Similarly, ``isomorphism'' will be consistently used in place of
``quasi-isomorphism'' and we will always assume that isomorphisms
preserve whatever additional structures of objects are implicit.
Given a scheme $X$, the symbol $\D(X)$ will denote a derived category
of sheaves of $\ms O_X$-modules, with various conditions (boundedness,
perfection, quasi-coherence of cohomology) clear from context.  In the
end, it will suffice to work in the category denoted
$\D_{\text{fTd}}(X)$ by Hartshorne in \cite{hartshorne-residues}, so
the hypotheses on $\D$ will not be a focus of attention.

\begin{defn} 
  Given a scheme $X$, an object $A\in \D(X)$ will be called a {\it
    weak $\ms O$-algebra\/} if there are maps $\mu:A\ltensor A\to A$
  and $i:\ms O\to A$ in $\D(X)$ which satisfy the usual axioms for an
  associative unital algebra, the diagrams being required to commute
  in the derived category.
\end{defn}
In other words, a weak algebra is an algebra object of the derived
category.  Note that the derived tensor product makes $\D(X)$ into a
symmetric monoidal additive category (as the universal property of
derived functors ensures that all different associations of an
iterated tensor product are naturally isomorphic).  Thus, it makes
sense to speak of ``associative'' algebra structures.

Given an additive symmetric monoidal category, one can define most of
the usual objects and maps of the theory of algebras: (unital)
modules, bimodules, linear maps, derivations, inner derivations, maps
of algebras, etc.  We leave it to the reader to write down precise
definitions of these terms, giving two examples: Given a map of weak
algebras $A\to B$, an {\it $\ms O$-linear derivation\/} from $A$ to
$B$ is a map $\delta:A\to B$ in $\D(X)$ such that
$\delta\circ\mu_A=\mu_B\circ(\id\ltensor\delta+\delta\ltensor\id)$ in
$\D(X)$.  A derivation from $A$ to $A$ is {\it inner\/} if there is an
$\alpha:\ms O\to A$ such that
$\delta=\mu\circ(\alpha\ltensor\id)-\mu\circ(\id\ltensor\alpha)$.

Given a ring map $\ms O\to\ms O'$, the derived functor
$(\cdot)\ltensor_{\ms O}\ms O':\D(\ms O)\to\D(\ms O')$ respects the
monoidal structure.  There results a natural base change operation for
weak algebras and modules.  (This operation will be consistently
written as a change of base on the right to avoid sign errors.)

Similarly, given a weak algebra $A$ and a left $A$-module $P$, the
functor $P\ltensor(\cdot)$ takes objects of $\D(\ms O)$ to
$A$-modules.  This follows from the natural associativity of the
derived tensor product.

The first non-trivial example of a weak algebra is given by
$$\rend(K):=\rhom(K,K)$$ for a
perfect complex $K$.  Replacing $K$ by a projective resolution, one
easily deduces the weak algebra structure from the usual composition
of functions: if we write $K$ as a finite complex of free modules
(which we will also call $K$), then $\rend(K)$ has as $n$th module
$\prod_p\hom(K^p,K^{p+n})$, with differential
$\partial^n(\alpha_p)_q=(-1)^{n+1}\alpha_{q+1}d+d\alpha_q$.  Since $K$
is perfect, the $n$th module of $\rend(K)\ltensor\rend(K)$ is equal to
$$\prod_{a+b=n}\prod_{s,t}\hom(K^s,K^{s+a})\tensor\hom(K^t,K^{t+b})$$ and the 
multiplication projects to the factors with $s=t+b$ and then composes
functions as usual.  Setting $K^{\vee}=\rhom(K,\ms O)$ (the derived
dual of $K$), we have the following basic lemma.

\begin{lem}\label{L:duality compatibility} Let $K$ be a perfect
  complex.
  \begin{enumerate}
  \item[$($i$)$] There is a natural isomorphism $K\ltensor
    K^{\vee}\simto\rend(K)$.
  \item[$($ii$)$] There is a natural left action of $\rend(K)$ on $K$.
  \end{enumerate}
  Tensoring the action
$$\rend(K)\ltensor K\to K$$ with $K^{\vee}$ on the right and using $($i$)$ 
yields the multiplication of $\rend(K)$.
\end{lem}
It is essential that the action be written on the left (when using the
standard sign convention for forming the total complex of a double
complex \cite[I.1.2.1]{illusie}, \cite[Appendix]{matsumura}) and that
$K^{\vee}$ be written on the right for the signs to work out.  These
kinds of sign sensitivities abound in the derived category and require
vigilance.

An algebra of the form $\rend(K)$ will be called a {\it derived
  endomorphism algebra\/}.  Our goal is to re-prove the classical
results about matrix algebras for derived endomorphism algebras of
perfect complexes.

\begin{notn}
  The symbols $P$ and $Q$ will always be taken to mean perfect
  complexes with a chosen realization as a bounded complex of finite
  free modules.  Thus, maps $P\to Q$ in the derived category will
  always come from maps of the ``underlying complexes'' (taken to mean
  the chosen realizations).  Similarly, $\rend(P)$ will have as chosen
  representative the complex constructed from the underlying complex
  of $P$ as above: $\rend(P)^n=\prod_t\hom(P^t,P^{t+n})$ with
  differential $\partial(\alpha_t)_s=(-1)^{n+1}\alpha_{s+1}
  d+d\alpha_s$.
\end{notn}
These conventions facilitate making certain basic arguments without
speaking of replacing the object by a projective resolution, etc., but
it is ultimately only important for this book-keeping reason; the
reader may ignore it without fear (until it is explicitly invoked!).

\begin{defn} Given $M\in\D(\ms O)$, the {\it annihilator of $M$\/} is
  the kernel $\ann(M)$ of the natural map from $\ms O$ to
  $\End_{\D(\ms O)}(M)$.  The quotient $\ms O/\ann(M)$ will be denoted
  by $\ms O_M$.
\end{defn}

Given an isomorphism $\psi:P\to Q(n)$, there is an isomorphism 
$\psi^{\ast}:\rend(P)\to\rend(Q)$ given by functorial conjugation by $\psi$ 
followed by the natural identification
of $\rend(Q(n))$ with $\rend(Q)$.  We will call this the {\it induced map\/}.  
The map $\psi^{\ast}$ may also be described as follows: under the natural 
identification of 
$\rend(P)$ with $P\ltensor P^{\vee}$, $\psi^{\ast}$ is identified with 
$\psi\ltensor(\psi^{\vee})^{-1}$.

\begin{thm}\label{T:the one} Let $P$ and $Q$ be non-zero perfect complexes of 
$\ms O$-modules.  If $\rend(P)\cong\rend(Q)$ as weak algebras, then there 
exists a unique $n$ such 
that the map $$\isom(P,Q(n))\to\isom(\rend(P),\rend(Q))$$ is surjective with 
each fiber a torsor under $\ms O_P^{\times}$.  If $P=Q$, then $n=0$ and the 
kernel is naturally a 
split torsor.
\end{thm}
\begin{cor}\label{C:ders}  The sequence
$$0\to \ms O_P\to\End(P)\to\Der(\rend(P))\to 0$$
is exact.  More generally, if $P$ and $Q(n)$ are isomorphic, then the
map
$$\hom(P,Q(n))\to\Der(\rend(P),\rend(Q))$$ is surjective 
with each fiber naturally a torsor under $\ms O_P$.
\end{cor}
\begin{proof} Apply Theorem \ref{T:the one} to $P[\eps]$ (as a complex over $\ms
  O[\eps]$) and look at automorphisms of the weak algebra $\rend_{\ms
    O[\eps]}(P[\eps])$ reducing to the identity modulo $\eps$.
\end{proof}

The proof of Theorem \ref{T:the one} will make use of the completion of $\ms 
O$ to lift the classical theorems on matrix algebras from the closed fiber by 
``infinitesimal
induction.''

\begin{prop}\label{P:field} If $\ms O$ is a field $k$ then Theorem \ref{T:the 
one} and Corollary \ref{C:ders} hold.
\end{prop}
\begin{proof}  
  The bounded derived category of $k$ is naturally identified with the
  category of $\Z$-graded finite $k$-modules by sending a complex to
  the direct sum of its cohomology spaces.  Given perfect complexes
  $P$ and $Q$, the algebra $\rend(P)$ (resp.\ $\rend(Q)$) is then
  identified with a matrix algebra, carrying the induced grading from
  the grading of the vector space $P$ (resp.\ $Q$), and an isomorphism
  from $\rend(P)\to\rend(Q)$ is identified with an isomorphism of
  matrix algebras which respects the gradings.  By the allowance of a
  shift, we may restrict our attention to graded spaces whose minimal
  non-zero graded piece is in degree 0; any reference in what follows
  to graded vector spaces will implicitly assume this hypothesis.
  (Of course the algebras involved will still carry
  non-trivial graded pieces with negative degrees.)

  Let $A$ be a graded matrix algebra of rank $n^2$ and $V$ and $W$ two
  graded $n$-dimensional vector spaces with non-trivial graded
  $A$-actions.  By the Skolem-Noether theorem, there is an $A$-linear
  isomorphism $\alpha:V\to W$.  We claim that $\alpha$ is graded.  To
  prove this, it suffices to show that given a non-zero vector $v\in
  V_0$, $\alpha(v)$ is in $W_0$ (because $V$ and $W$ are simple
  $A$-modules).  Write $\alpha(v)=\sum w_i$.  Since $V$ is a simple
  $A$- module, $A_n\cdot v = V_n$; a similar statement holds for $W$
  (given a choice of non-zero weight 0 vector, which exists by the
  hypothesis on the gradings).  Thus, the highest non-trivial grading
  $N$ of $A$ will equal the highest non- trivial grading of both $V$
  and $W$.  Furthermore, given any $i$ such that $w_i\neq 0$, the fact
  that $A_{-i}\cdot w_i = W_0$ means that $A_{-i}\neq 0$.  Given $i>0$
  such that $w_i\neq 0$, we have for all $\tau\in A_{-i}$ that
$$0=\alpha(0)=\alpha(\tau(v))=\tau(\alpha(v))=\tau\Big(\sum 
w_j\Big)=\tau(w_i)+\textrm{higher terms.}$$ Thus, $\tau(w_i)=0$, which
implies that $W_0=0$.  This contradicts the assertion that $W_0$ is
the minimal non-trivial graded piece.  So $w_i=0$ for all $i>0$ and
therefore $w\in W_0$.  Translating this back into the derived
language, we have proven that given an isomorphism
$\phi:\rend(P)\to\rend(Q)$, there is an isomorphism $P\to Q$ in
$\D(k)$ which induces $\phi$ by functoriality.  In fact, we have shown
the rest of the proposition as well, because $\alpha$ is the unique
choice for such an isomorphism up to scalars by the Skolem-Noether theorem.

To prove Corollary \ref{C:ders}, let $V=\oplus V_i$ be a graded vector space and
$T\in\End(V)$ a non-central linear transformation.  We wish to show
that if the (non-trivial) inner derivation by $T$ is homogeneous of
degree $0$ then $T$ is homogeneous of degree $0$.  To do this,
consider the restriction of $T$ to the degree $0$ part of $\End(V)$.
Let $T^{n}$ be a graded component of $T$ (so that
$T^{n}:V\to V$ shifts degrees by $n$).  Let $V^{m}$ be a graded
component such that the induced transformation
$T^{n}:V^{m}\to V^{m+n}$ is non-zero.  Consider the graded linear
transformation (of degree $0$) $S:V\to V$ which acts as the identity
on $V^m$ and annihilates every other component.  It is easy to see
that the commutator $[T^{n},S]$ is $T^{n}S$, which implies that
$[T,S]$ has a non-trivial component of degree $n$.  Since the
derivation $[T,-]$ is homogeneous of degree $0$, it follows that
$n=0$, and thus $T$ is homogeneous of degree $0$, as desired.
\end{proof}

\begin{lem}\label{L:completing} 
  Theorem \ref{T:the one} is true for $\ms O$ if it is true for
  $\widehat{\ms O}$.
\end{lem}
\begin{proof} 
  We proceed by reducing the problem to a question of linear algebra
  and then using the faithful flatness of completion.

  Suppose given $P$ and $Q$ and an isomorphism
  $\phi:\rend(P)\to\rend(Q)$; this defines an action of $A:=\rend(P)$
  on $Q$.  We claim that finding $u:P\to Q$ such that $\phi=u^{\ast}$
  is equivalent to finding an $A$-linear isomorphism from $P$ to $Q$.
  Indeed, suppose $u:P\to Q$ is $A$-linear, so that the diagram
$$\xymatrix{\rend(P)\ltensor P\ar[r]^{\phi\ltensor u}\ar[d] & 
  \rend(Q)\ltensor Q\ar[d]\\
  P\ar[r]^u & Q }$$ commutes, where the vertical arrows are the
actions.  Tensoring the left side with $P^{\vee}$ and the right side
with $Q^{\vee}$, we see that the resulting diagram
$$\xymatrix{\rend(P)\ltensor P\ltensor P^{\vee}\ar[rr]^{\phi\ltensor u\ltensor 
(u^{\vee})^{-1}}\ar[d] & & \rend(Q)\ltensor Q\ltensor Q^{\vee}\ar[d]\\
P\ltensor P^{\vee}\ar[rr]^{u\ltensor (u^{\vee})^{-1}} & & Q\ltensor Q^{\vee}
}$$
also commutes.  Applying Lemma \ref{L:duality compatibility}
and writing $B$ for $\rend(Q)$, we find that the diagram
$$\xymatrix{A\ltensor A\ar[r]^{u^{\ast}\ltensor\phi}\ar[d] & B\ltensor 
  B\ar[d]\\
  A\ar[r]^{u^{\ast}} & B }$$ commutes, where the vertical arrows are
the multiplication maps.  Considering the units in the algebras, one
readily concludes the proof of the claim.  Note that to conclude that
any such $u$ as above is an isomorphism, it suffices for its reduction
to the residue field to be an isomorphism (e.g.\ because the complexes
are bounded above).

It is easy to see (using the realization in terms of diagrams of
finite flat $\ms O$-modules) that $\hom_{\D(\ms O)}$ is compatible
with flat base change and completion when restricted to the category
of perfect complexes: given a flat ring extension $\ms O\to\ms O'$,
there is a natural isomorphism $$\hom_{\D(\ms O')}(M\ltensor\ms
O',N\ltensor\ms O')\cong\hom_{\D(\ms O)}(M,N)\tensor\ms O'$$ for all
perfect $M$ and $N$ in $\D(\ms O)$.  Furthermore, given a perfect weak
algebra $\Xi$, the realization of the module of $\Xi$-linear maps as a
kernel of maps of $\hom$-modules shows that the same statement is true
for $\hom_{\Xi}$.  Thus there is a commutative diagram
$$\xymatrix@C=0pt{\hom_{\Xi}(M,N)\ar[d]\ar[rr] & & 
  \hom_{\widehat{\Xi}}(\widehat M,\widehat N)\ar[d]\\
  \hom_{\Xi}(M,N)\tensor_{\ms O}k\ar@{=}[rr]\ar[dr] & &
  \hom_{\widehat{\Xi}}(\widehat M,\widehat N)\tensor_{\widehat{\ms
      O}}k\ar[dl]\\
  & \hom_{\Xi\ltensor k}(M\ltensor k,N\ltensor k).  }$$ with
surjective vertical arrows.  This immediately applies to our situation
to show that the map of Theorem \ref{T:the one} is surjective for $\ms O$
if it is for $\widehat{\ms O}$ (for a fixed $n$, which may be determined
from the reduction to the residue field).  Indeed, a
$\widehat{\Xi}$-linear map $\widehat M\to\widehat N$ yields an element
of $\hom_{\widehat{\Xi}}(\widehat M,\widehat N)\tensor k$ whose image
in the bottom module is an isomorphism.  It follows from the diagram
that there is a $\Xi$-linear map $u:M\to N$ whose (derived) reduction
to $k$ is an isomorphism, whence $u$ is an isomorphism by Nakayama's
lemma for perfect complexes (see e.g.\ Lemma 2.1.3 of \cite{mod-of-comp}).

Similarly, to verify that an isomorphism $\xi:P\simto P$ in the kernel
of the automorphism map is homotopic to a constant, it suffices to
show that an element $\xi\in\End_{\D(\ms O)}(P)$ is in $\ms O_P$ if
and only if this is true after completing.  But the module of maps
homotopic to a constant is also clearly compatible with flat base
change and completion is moreover {\it faithfully\/} flat (all modules
involved are finite over $\ms O$ because the complexes involved are
perfect), so $\xi$ is in a submodule $Z$ of $\End(P)$ if and only if
its image in $\End(P)\tensor\widehat{\ms O}$ is contained in
$Z\tensor\widehat{\ms O}$.
\end{proof}

From this point onward, {\it we assume that $\ms O$ is a complete
  local Noetherian ring\/}.  Recall that a quotient of local rings
$0\to I\to \ms O\to\widebar{\ms O}\to 0$ is {\it small\/} if $I$ is
generated by an element $\eps$ which is annihilated by the maximal
ideal of $\ms O$ (so that, in particular, $\eps^2=0$).  We can choose
a filtration $\ms O\supset\mf m=I_0\supset I_1\supset
I_2\supset\cdots$ which is separated (i.e., so that $\cap_{i}
I_{i}=0$) and defines a topology equivalent to the $\mf m$-adic
topology such that for all $i\geq 0$, the quotient $0\to
I_i/I_{i+1}\to\ms O/I_{i+1}\to\ms O/ I_i\to 0$ is a small extension,
with $I_i/I_{i+1}$ generated by $\eps_i$.  We fix such a filtration
for remainder of this section, and we denote $\ms O/I_n$ by $\ms O_n$.

\begin{lem}\label{L:algebra quotient} 
  Let $0\to I\to R\to\widebar{R}\to 0$ be a surjection of rings.  Let
  $A$ be a weak $R$-algebra and $P$ and $Q$ two left $A$-modules.  Let
  $T$ denote the triangle in $\D(R)$ arising from the quotient map
  $R\to \widebar{R}$.
\begin{enumerate}
\item[$(i)$] The maps in $P\ltensor T$ are $A$-linear $($with the
  natural $A$-module structures$)$.
\item[$(ii)$] Any $A$-linear map $\psi:P\to Q\ltensor\widebar{R}$
  factors through an $A$-linear map $\widebar{\psi}:
  P\ltensor\widebar{R}\to Q\ltensor\widebar{R}$ which is the derived
  restriction of scalars of an $A\ltensor\widebar{R}$-linear map from
  $P\ltensor\widebar{R}$ to $Q\ltensor\widebar{R}$.
\item[$(iii)$] If $R\to\widebar{R}$ is a small extension of local rings with 
residue field $k$, then the
natural identification $P\ltensor I\simto P_k$ induced by a choice of 
basis for $I$ over $k$ is $A$-linear.
\end{enumerate}
\end{lem}
\begin{proof} 
  Note that basic results about homotopy colimits allow us to replace
  any object of $\D(R)$ by a complex of projectives, so there are no
  boundedness conditions on any of the complexes involved.  Part (i)
  follows immediately from the fact that $P\ltensor(\cdot)$ is a
  functor from $\D(R)$ to $A$-modules.  Part (ii) follows from writing
  $P$ and $A$ as complexes of projectives and representing the map
  $P\to Q\ltensor\widebar{R}$ as a map on complexes.  (Note that this
  factorization need not be unique as a map in $\D(R)$, but it is
  unique as the derived restriction of scalars from a map in
  $\D(\widebar R)$.)  Part (iii) follows similarly from looking at
  explicit representatives of $P$ and $A$.
\end{proof}

\begin{lem}\label{L:homotopy} 
  Suppose $f,g:P\to Q$ are two maps of perfect complexes in $K(\ms
  O)$.  Let $P_n=P\tensor\ms O_n$, $Q_n=Q\tensor\ms O_n$,
  $f_n=f\tensor\ms O_n$, $g_n=g\tensor\ms O_n$.  Suppose there are
  homotopies
$$h(n)\in\prod_t\hom(P^t,Q^{t-1}\tensor I_n)$$ such that for all $n$,
$$f_{n}-g_{n}=d\Big(\sum_{s<n}\bar h(s)\Big)+\Big(\sum_{s<n}\bar h(s)\Big)d$$
as maps of complexes, where $\bar h$ denotes the induced map.  Then
$f$ is homotopic to $g$.
\end{lem}
\begin{proof} The element $h=\sum_{s=0}^{\infty} h(s)$ converges and defines 
the homotopy.
\end{proof}

\begin{lem}\label{L:auto} 
  Let $0\to I\to R\to\widebar{R}\to 0$ be a small extension of local
  rings with residue field $k$.  Let $P$ and $Q$ be perfect complexes
  of $R$-modules $($with chosen realizations$)$ and
  $\phi:\rend(P)\to\rend(Q)$ an isomorphism of the derived
  endomorphism algebras, written as a map in that direction on the
  underlying complexes.  If there exists an isomorphism of the
  underlying complexes $\widebar u:\widebar P\simto\widebar Q$ such
  that $\widebar{\phi}={\widebar u}^{\ast}$ as maps of complexes, then
  there is a lift $u$ of $\widebar u$ and a homotopy $h$ between
  $\phi$ and $u^{\ast}$ such that $h(\rend(P))\subset \rend(Q)\tensor
  I$.  In particular, $\phi=u^{\ast}$ in $\D(R)$.
\end{lem}
\begin{proof} 
  Let $A=\rend(P)$ and let $A$ act on $Q$ via $\phi$.  The
  identification of $\widebar{\phi}$ with ${\widebar u}^{\ast}$
  provides an $\widebar{A}$-linear isomorphism $\widebar{\gamma}:
  \widebar P\to\widebar{Q}$, and we wish to lift this to an $A$-linear
  isomorphism $P\to Q$.  Consider the composition $P\to \widebar{Q}\to
  Q\tensor I(1)\cong Q_k(1)$ in the derived category.  By Lemma
  \ref{L:algebra quotient}, this map is $A$-linear and factors through
  an $A$-linear map $\alpha:P_k\to Q_k(1)$ which comes by derived
  restriction of scalars from an $A_k$-linear map in $\D(k)$.  By
  Proposition \ref{P:field} (and the method of its proof), we see that
  $\alpha$ is either zero or an isomorphism.  But $P_k\cong
  Q_k\not\cong 0$, which implies that $\alpha=0$.  This means that
  there is an {\it $R$-linear\/} lift $\gamma$ of $\widebar{\gamma}$.
  Now $(\gamma^{\ast})^{-1}\circ\phi-\id$ is identified with a map
  $\rend_k(P_k)\to\rend_k(P_k)$ in $\D(k)$ which is a derivation of
  the algebra, hence is homotopic to the inner derivation induced by a
  map $\omega_k:P_k\to P_k$ in $\D(k)$.  Writing $\omega$ for the
  composition
$$\xymatrix{P\ar[r] &
P_k\ar[r]^{\omega_k}& P_k\ar[r]^{\cong}& P\ltensor I\ar[r]& 
P,}$$
we see that there 
is a homotopy between
$\phi$ and $\gamma(1+\omega)^{\ast}$ with image in $\rend(Q)\tensor I$, and 
that $\gamma(1+\omega)$ is a lift of $\widebar{\gamma}$ as maps of complexes. 
\end{proof}

\begin{lem}\label{L:kernel} 
  Let $0\to I\to R\to\widebar{R}\to 0$ be a small extension of local
  rings with residue field $k$.  Let $P$ be a perfect complex of
  $R$-modules $($with a chosen realization$)$ and $\psi:P\to P$ an
  automorphism of the underlying complex such that
  $\widebar{\psi}=\widebar{\alpha}$ for some
  $\widebar{\alpha}\in\widebar R_{\widebar P}$ as maps of the complex
  $\widebar P$ and such that $\psi^{\ast}$ is homotopic to the
  identity as a map of weak algebras.  Then there is a unit $\alpha$
  lifting $\widebar{\alpha}$ and a homotopy $h$ between $\psi$ and
  $\alpha$ such that $h(\rend(P))\subset \rend(P)\tensor I$.
\end{lem}
\begin{proof} 
  The proof is quite similar to the proof of Lemma \ref{L:auto}, using
  the left half of the exact sequence of Corollary \ref{C:ders} rather
  than the right half.
\end{proof}

\begin{prop}\label{P:done}  
  Theorem \ref{T:the one} holds for $\ms O$ $($now assumed
  complete$)$.
\end{prop}
\begin{proof} 
  Given an isomorphism $\phi:\rend(P)\simto\rend(Q)$, we may assume
  after adding zero complexes to $P$ and $Q$, shifting $Q$, and
  applying a homotopy to $\phi$, that there is an isomorphism
  $\psi_0:P_0\to Q_0$ such that $\phi_0=\psi_0^{\ast}$ as maps of
  complexes.  We can now apply Lemma \ref{L:auto} to arrive at an
  isomorphism $\psi_1$ lifting $\psi_0$ and a homotopy $\widebar h(0)$
  with image in $\rend_{\ms O_1}(Q_1)\tensor_{\ms O_1}I_0/I_1$ between
  $\phi_1$ and $\psi_1^{\ast}$.  Lift $\widebar h(0)$ to a homotopy
  $h(0)$ with image in $\rend(Q)\tensor I_0$.  Then $(\phi-
  (dh(0)+h(0)d))_1=\psi_1^{\ast}$ as maps of complexes, and we may
  find a homotopy $h(1)$, etc.  By Lemma \ref{L:homotopy}, we see that
  there is an isomorphism $\psi:P\to Q$ such that $\phi=\psi^{\ast}$
  in $\D(\ms O)$.  A similar argument shows that the kernel is $\ms
  O_P^{\ast}$.
\end{proof}

\subsection{The construction of $\GAz$}\label{S:gen azumaya construction}
In this section, we define a stack which we will use to compactify the
stack of Azumaya algebras.  While the definition is rather technical
in general, in the case of a relative surface it assumes a simpler and
more intuitive form.  

Let $(X,\ms O)$ be a ringed topos.

\begin{defn}  
  A \emph{pre-generalized Azumaya algebra\/} on $X$ is a perfect
  algebra object $\ms A$ of the derived category $\D(X)$ of
  $\ms O_{X}$-modules such that there exists an object $U\in X$ covering
  the final object and a totally supported perfect sheaf $\ms F$ on
  $U$ with $\ms A|_{U}\cong\rsend_{U}(\ms F)$ as weak algebras.  An
  isomorphism of pre-generalized Azumaya algebras is an isomorphism in
  the category of weak algebras.
\end{defn}

\subsubsection{Stackification}
\label{sec:stackification}

  Consider the fibered category $\PR\to Sch_{\etale}$ of
  pre-generalized Azumaya algebras on the large \'etale topos over
  $\spec\Z$.  We will stackify this to yield the stack of generalized
  Azumaya algebras.  This is slightly different from the construction
  given in \cite[3.2]{l-mb}, as we do not assume that the
  fibered category is a pre-stack.

\begin{lem}\label{L:stackification} 
  Suppose $T$ is a topos and $\ms C\to T$ is a category fibered in
  groupoids.  There exists a stack $\ms C^{s}$, unique up to
  1-isomorphism, and a 1-morphism $\ms C\to\ms C^{s}$ which is
  universal among 1-morphisms to stacks (up to 2-isomorphism).
\end{lem}
\begin{proof} 
  The proof is the usual type of argument.  A reader interested in
  seeing a generalization to stacks in categories larger than
  groupoids should consult \cite{giraud}.  First, we may assume that
  in fact $\ms C\to T$ admits a splitting (after replacing $\ms C$ by
  a 1-isomorphic fibered category).  Define a new fibered category
  $\ms C^{p}$ as follows: the objects will be the same, but the
  morphisms between two objects $a$ and $b$ over $t\in T$ will be the
  global sections of the sheafification of the presheaf
  $\hom_{t}(a,b):(s\xto{\phi}
  t)\mapsto\hom_{s}(\phi^{\ast}a,\phi^{\ast}b)$ on $t$.  Clearly $\ms
  C^{p}$ is a prestack (i.e., given any two sections $a$ and $b$ over
  $t$, the hom-presheaf just described is a sheaf) and the natural map
  $\ms C\to\ms C^{p}$ of fibered categories over $T$ is universal up
  to 1-isomorphism for 1-morphisms of $\ms C$ into prestacks.  Now we
  apply \cite[3.2]{l-mb} to construct $\ms C^{s}$ as the
  stackification of $\ms C^{p}$.
\end{proof}

\begin{defn} 
  The stack of \emph{generalized Azumaya algebras\/} on schemes is
  defined to be the stack in groupoids $\PR^{s}\to Sch_{\etale}$
  associated to the fibered category of pre-generalized Azumaya
  algebras.
\end{defn}

\begin{remark} 
  Explicitly, given a scheme $X$, to give a generalized Azumaya
  algebra on $X$ is to give an \'etale 3-hypercovering
  $\xymatrix{Y''\ar@<4pt>[r]\ar[r]\ar@<-4pt>[r]&Y'\ar@<2pt>[r]\ar@<-2pt>[r]&Y\ar[r]&X}$,
  a totally supported sheaf $\ms F$ on $Y$, and a gluing datum for
  $\rsend_{Y}(\ms F)$ in the derived category $\D(Y')$ whose
  coboundary in $\D(Y'')$ is trivial.  Two such objects $(Y_{1},\ms
  F_{1},\delta_{1})$ and $(Y_{2},\ms F_{2},\delta_{2})$ are isomorphic
  if and only if there is a common refinement $Y_{3}$ of the
  3-hypercovers $Y_{1}$ and $Y_{2}$ and an isomorphism $\phi:\ms
  F_{1}|_{Y_{3}}\simto\ms F_{2}|_{Y_{3}}$ commuting with the
  restrictions of $\delta_{1}$ and $\delta_{2}$.  Thus, a generalized
  Azumaya algebra is gotten by gluing ``derived endomorphism
  algebras'' together in the \'etale topology.  When $X$ is a
  quasi-projective smooth surface, or, more generally, a
  quasi-projective scheme smooth over an affine with fibers of
  dimension at most 2, then the sections of $\PR$ over $X$ are the
  same as the sections of $\PR^{s}$ over $X$; see Section \ref{S:p=s on
    surface}.
\end{remark}
\begin{example} 
  Let $\pi:\ms X\to X$ be a $\m_{n}$-gerbe and $\ms F$ a totally
  supported perfect $\ms X$-twisted sheaf.  The complex
  $\R\pi_{*}\R\send_{\ms X}(\ms F)\in\D(X)$ is a pre-generalized
  Azumaya algebra, hence has a naturally associated generalized
  Azumaya algebra.  We will see below that the global sections of the
  stack $\PR^{s}$ are precisely the weak algebras of this form.
\end{example}

\begin{lem}\label{L:isoms-check}
  Let $\ms F$ and $\ms G$ be totally supported perfect sheaves on a
  $\G_m$-gerbe $\ms X\to X$.
  \begin{enumerate}
  \item The sheaf of isomorphisms between the generalized Azumaya
    algebras associated to the weak algebras $\R\pi_{\ast}\rsend(\ms
    F)$ and $\R\pi_{\ast}\rsend(\ms G)$ is naturally isomorphic to
    $\isom(\ms F,\ms G)/\G_m$, with $\G_m$ acting by scalar
    multiplication on $\ms G$.
  \item Any isomorphism of generalized Azumaya algebras
    $\phi:\R\pi_{\ast}\rsend(\ms F)\simto\R\pi_{\ast}\rsend(\ms G)$ is
    the isomorphism associated to an isomorphism $\ms F\simto\ms
    L\tensor\ms G$ for some invertible sheaf $\ms L$ on $X$.
  \end{enumerate}
\end{lem}
\begin{proof}
  Temporarily write $\ms I$ for the sheaf of isomorphisms of
  generalized Azumaya algebras from $\R\pi_{\ast}\rsend(\ms F)$ to
  $\R\pi_{\ast}\rsend(\ms G)$.  There is clearly a map $\chi:\isom(\ms
  F,\ms G)/\G_m\to\ms I$.  To verify that it is an isomorphism, it
  suffices to verify it \'etale-locally on $X$, whence we may assume
  that $X$ is strictly local.  Choosing an invertible $\ms X$-twisted
  sheaf and twisting down $\ms F$ and $\ms G$, we are reduced to
  showing the analogous statement for totally supported sheaves on $X$
  itself. Any local section of $\ms I$ comes from an isomorphism of
  weak algebras $\rsend(\ms F)\simto\rsend(\ms G)$, so Theorem \ref{T:the one}
  shows that $\chi$ is surjective.  Similarly, any section of the
  kernel of $\chi$ must locally be trivial, whence $\chi$ is an
  isomorphism.

  The second part is a formal consequence of the first: there is an
  \'etale covering $U\to X$ such that $\phi|_U$ is associated to an
  isomorphism $\psi:\ms F|_U\simto\ms G|_U$.  The coboundary of $\psi$
  on the product $U\times_X U$ is multiplication by some scalar, which
  is a cocycle by a formal calculation.  This gives rise to the
  invertible sheaf $\ms L$; tensoring with $\ms L$ makes the
  coboundary of $\psi$ trivial, whence it descends to an isomorphism
  $\ms F\to\ms L\tensor\ms G$ inducing $\phi$, as desired.
\end{proof}

\begin{defn}\label{D:gerbe-o-trivs} 
  Let $\ms A$ be a generalized Azumaya algebra on $X$.  The
  \emph{gerbe of trivializations\/} of $\ms A$, denoted $\ms X(\ms
  A)$, is the stack on the small \'etale site $X_{\etale}$ whose sections
  over $V\to X$ given by pairs $(\ms F,\phi)$, where $\ms F$ is a
  totally supported sheaf on $V$ and $\phi:\R\send_{V}(\ms F)\simto\ms
  A|_{V}$ is an isomorphism of generalized Azumaya algebras.  The
  isomorphisms in the fiber categories are isomorphisms of the sheaves
  which respect the identifications with $\ms A$, as usual.
\end{defn}
This is entirely analogous to the gerbe produced in section V.4.2 of
\cite{giraud}.  There is also an analogue of the $\m_n$-gerbe
associated to an Azumaya algebra of degree $n$.

\begin{defn}
  Given a generalized Azumaya algebra $\ms A$ of degree $n$ on $\ms
  X$, the \emph{gerbe of trivialized trivializations\/} of $\ms A$,
  denoted $\ms X_{\triv}(\ms A)$, is the stack on the small \'etale
  site of $X$ whose sections over $U\to X$ consist of triples $(\ms
  F,\phi,\delta)$ with $\phi:\R\send_{U}(\ms F)\simto\ms A_U$ an
  isomorphism of generalized Azumaya algebras and $\ms
  O_U\simto\det(\ms F)$ an isomorphism of invertible sheaves on $U$.
  The isomorphisms in the fiber categories are isomorphisms of the
  sheaves which preserve the identifications with $\ms A$ and the
  trivializations of the determinants.
\end{defn}

\begin{lem} The stack $\ms X(\ms A)$ is a $\G_{m}$-gerbe.  If $\ms A$
  has degree $n$, then $\ms X_{\triv}(\ms A)$ is a $\m_n$-gerbe whose
  associated cohomology class maps to $[\ms X(\ms A)]$ in
  $\H^2(X,\G_m)$.
\end{lem}
\begin{proof} 
  This follows immediately from the derived Skolem-Noether Theorem
  \ref{T:the one} and the fact that all of the sheaves $\ms F$ are
  totally supported.
\end{proof}

\begin{cor} A generalized Azumaya algebra $\ms A$ has a class in
  $\H^{2}(X,\G_{m})$.  When the rank of $\ms A$ is $n^{2}$, $\ms A$
  has a class in $\H^{2}(X,\m_{n})$ (in the fppf topology).
\end{cor}

\begin{defn} When $\rk\ms A=n^{2}$, we will call the cohomology class
  in $\H^{2}(X,\m_{n})$ the \emph{class of $\ms A$\/}, and write
  $\cl(\ms A)$.
\end{defn}

Let $\pi:\ms X(\ms A)\to X$ denote the natural projection.
\begin{lem}\label{L:embedding lemma} 
  There is an $\ms X(\ms A)$-twisted sheaf $\ms F$ and an isomorphism
  of generalized Azumaya algebras $\phi:\R\pi_{\ast}\R\send_{\ms X(\ms
    A)}(\ms F)\simto\ms A$.  The datum $(\ms X(\ms A),\ms F,\phi)$ is
  functorial in $\ms A$.
\end{lem}
\begin{proof} 
  As usual, the construction of $\ms X(\ms A)$ yields by first
  projection a twisted sheaf $\ms F$.  Whenever $\ms X(\ms A)$ has a
  section $f$ over $V$, there is an isomorphism
  $\R\send_{V}(f^{\ast}\ms F)\to\ms A|_{V}$ by construction, and this
  is natural in $V$ and $f$.  This is easily seen to imply the
  remaining statements of the lemma.
\end{proof}

Let $\ms D\to Sch_{\etale}$ denote the fibered category of derived
categories which to any scheme $X$ associates the derived category
$\D(X)$ of \'etale $\ms O_{X}$-modules.
 
\begin{prop}\label{P:gen az isom} 
  There is a faithful morphism of fibered categories $\PR^{s}\to\ms
  D$ which identifies $\PR^{s}$ with the subcategory of $\ms D$
  whose sections over $X$ are weak algebras of the form
  $\R\pi_{\ast}\R\send_{\ms X}(\ms F)$, where $\pi:\ms X\to X$ is a
  $\G_{m}$-gerbe, and whose isomorphisms $\R\pi_{\ast}\R\send_{\ms
    X}(\ms F)\simto\R\pi'_{\ast}\R\send_{\ms X'}(\ms F')$ are
  naturally a pseudotorsor under $\saut(\ms F)/\G_m$.  
\end{prop}
\begin{proof} 
  The morphism $\PR^{s}\to\ms D$ comes from Lemma \ref{L:embedding lemma}.
  Given $\ms A$ and $\ms B$, an isomorphism $\phi:\ms A\to\ms B$
  induces an isomorphism $\ms X(\ms A)\simto\ms X(\ms B)$.  Thus,
  given $\ms X,\ms F,\ms X',\ms F'$, an isomorphism
  $\R\pi_{\ast}\rsend(\ms F)\simto\R\pi'_{\ast}\rsend(\ms F')$ induces
  an isomorphism $\eps:\ms X\to\ms X'$ and an isomorphism of
  generalized Azumaya algebras $\R\pi_\ast\rsend(\ms
  F)\simto\R\pi_\ast\rsend(\eps^{\ast}\ms F')$.  (In particular, any
  isomorphism is identified with an isomorphism of the underlying weak
  algebras.)  By Lemma \ref{L:isoms-check}, once there is an isomorphism the
  set of isomorphisms is a torsor under $\saut(\ms F)/\G_m$, as
  claimed.  The faithfulness also results from Lemma \ref{L:isoms-check}.
\end{proof}

\begin{remark}\label{R:rigid-prep}
  When $\ms X=\ms X'$ in Proposition \ref{P:gen az isom}, the sheaf of
  isomorphisms is simply identified with the quotient sheaf $\isom(\ms
  F,\ms F')/\G_m$.  This will be the case when we study the moduli of
  generalized Azumaya algebras on a surface, as the (geometric) Brauer
  class will be constant in families.
\end{remark}
 
Thus, at the end of the complex process of stackification, one is 
left simply with the derived endomorphism algebras of twisted sheaves, 
with a subset of the quasi-isomorphisms giving the isomorphisms.

\subsubsection{Identification with rigidifications}\label{sec:riggaz}
Let $\ms G_{X}(n)$ be the stack of generalized Azumaya algebras on $X$
of degree $n$.

\begin{prop}\label{P:G-rig}
  The morphism $\rho:\ms T_{X}^{\parf}\to\ms G_{X}$ sending $\ms F$ to
  $\rsend(\ms F)$ yields an isomorphism $\ms
  T_{X}\thickslash\G_m\simto\ms G_{X}$.
\end{prop}
\begin{proof}
  It follows from Theorem \ref{T:the one} that $\ms G_{X}$ is the
  stackification of the prestack given by taking totally supported
  sheaves and replacing $\isom(\ms F,\ms G)$ with $\isom(\ms F,\ms
  G)/\G_m$.  But this is precisely how $\ms T_{X}^{\parf}\thickslash \G_m$
  is constructed!
\end{proof}

\begin{lem}\label{L:stupid-lemma}
  If $\ms Q\to\ms Q'$ is a morphism of prestacks on a site which is
  fully faithful on fiber categories and an epimorphism (i.e., any
  object of $\ms Q'$ is locally in the image of $\ms Q$) then the
  induced map of stackifications is an isomorphism.
\end{lem}
\begin{proof}
  An object of the stackification is just an object of the prestack
  with a descent datum.  Moreover, refining the descent datum yields a
  naturally isomorphic object of the stackification.  Thus, we see
  that the map on stackifications $\widetilde{\ms Q}\to\widetilde{\ms
    Q'}$ is fully faithful and an epimorphism.  (Indeed, after
  refining the descent datum on an object of $\ms Q'$, we can assume
  the object and descent datum come from $\ms Q$.)  It follows that it
  must be an isomorphism.
\end{proof}

\begin{prop}\label{P:G-rig-det}
  Assume $n$ is invertible on $X$.  The morphism $\rho:\ms
  T_{X}^{\ms O}(n)\to\ms G_{X}(n)$ sending $\ms F$ to $\rsend(\ms
  F)$ yields an isomorphism $\ms T_{X}^{\ms
    O}\thickslash\m_n\simto\ms G_{X}$.
\end{prop}
\begin{proof}
  In light of Proposition \ref{P:G-rig}, it is enough to prove that the natural
  map $\phi:\ms T^{\ms O}_{X}(n)\to\ms T_{X}(n)$ yields an
  isomorphism of the appropriate rigidifications.  Clearly, $\phi$ is
  an epimorphism.  Moreover, for any $\ms F$ and $\ms G$ with
  trivialized determinants, $\phi$ induces an isomorphism of sheaves
  $f:\isom_{\det}(\ms F,\ms G)/\m_n\simto\isom(\ms F,\ms G)/\G_m$.  To
  check this, it is enough to suppose $X=\spec A$ is strictly
  Henselian.  Since $n$ is invertible on $X$, any unit of $A$ has an
  $n$th root, from which it follows that $f$ is surjective.  If
  $\gamma:\ms F\to\ms G$ and $\eta:\ms F\to\ms G$ are isomorphisms
  which preserve determinants and differ by multiplication by a scalar
  $\theta$ on $\ms G$ then $\theta$ must be an $n$th root of unity,
  which shows that $f$ is injective.

  Forming prestacks by dividing out by the appropriate scalars, we
  thus find a morphism of prestacks $\ms Q\to\ms Q'$ which is fully
  faithful on fiber categories and is an \'etale epimorphism.
  Applying Lemma \ref{L:stupid-lemma} completes the proof.
\end{proof}

\subsubsection{The relative case}

Now we push everything forward (with one important Warning
\ref{W:warning} below) to define relative stacks of generalized Azumaya
algebras.

\begin{defn} 
  Let $f:X\to S$ be a morphism.  A \emph{relative generalized Azumaya
    algebra\/} on $X/S$ is a generalized Azumaya algebra on $X$ whose local
  sheaves are $S$-flat and totally pure in each geometric fiber.
  This is equivalent to writing $\ms A\cong\R\pi_{\ast}\R\send_{\ms
    X}(\ms F)$ with $\ms X\to X$ a $\G_{m}$-gerbe and $\ms F$ an
  $S$-flat $\ms X$-twisted sheaf which is totally pure in every
  geometric fiber.
\end{defn}
\begin{warning}\label{W:warning}
  Even though the absolute theory of generalized Azumaya algebras used
  totally supported sheaves, in the relative theory we will use
  totally pure sheaves.  While this is not necessary for the abstract
  results to be true, it gives a better moduli theory when $X/S$ is
  sufficiently nice (e.g., a smooth projective surface).
\end{warning}

As in Definition \ref{D:gerbe-o-trivs}, one may define the class of such a
generalized Azumaya algebra.  Let $\ms X\to X$ be a fixed
$\m_{n}$-gerbe, with $n\in\ms O_{S}(S)^{\times}$.

\begin{notn} 
  Let $\GAz_{\ms X/S}(n)$ denote the stack of generalized Azumaya
  algebras on $X/S$ of rank $n^{2}$ in every geometric fiber whose
  class agrees with $[\ms X]$ \'etale locally around every point on
  the base.  When we do not wish to specify the cohomology class, we
  will write $\GAz_{X/S}(n)$ for the stack of generalized Azumaya
  algebras of rank $n^2$ on each fiber.
\end{notn}
When $X\to S$ is proper and $n$ is invertible on $S$, the condition
that the cohomology class agree with $[\ms X]$ \'etale-locally on $S$
is equivalent to the condition that it agree with $[\ms X]$ in every geometric
fiber.

\subsubsection{Identification of $\GAz_{\ms X/S}(n)$ with $\ms M^{\ms
    X}_n$}
\label{sec:comp}

Let $\ms G_{X/S}(n)$ be the stack on $X_{\retale}$ parametrizing
generalized Azumaya algebras which are locally isomorphic to
$\rsend(\ms F)$ with $\ms F$ an object of $\ms P_{X/S}^{\parf}(n)$.  An
argument identical to Proposition \ref{P:G-rig} shows that the natural map $\ms
P_{X/S}^{\ms O}(n)\to\ms G_{X/S}(n)$ yields an isomorphism $\ms P^{\ms
  O}_{X/S}(n)\thickslash\m_n\simto\ms G_{X/S}(n)$.

On the other hand, it is easy to see that $\GAz_{X/S}(n)=f_{\ast}(\ms
G_{X/S}(n))$ and $\GAz_{\ms X/S}(n)=f_{\ast}^{\ms X}(\ms
G_{X/S}(n))$.  We conclude that $\GAz_{\ms X/S}(n)\cong\ms M^{\ms
  X}_n$, thus showing that generalized Azumaya algebras give a
coherent model for $\ms M^{\ms X}_n$.  Moreover, it is easy to see
that $\ms M^{\ms X}_n(c)$ is identified with the stack of generalized
Azumaya algebras of the form $\rshom(\ms F)$ with $\ms F$ an $\ms
X$-twisted sheaf with trivialized determinant and $\deg c_2(\ms
F)=c/2n$.  This condition is equivalent to the condition that $\deg
c_2(\rsend(\ms F))=c$, giving a coherent interpretation of $\ms M^{\ms
  X}_n(c)$.

For the reader uncomfortable with the stackification procedure (in
spite of its concrete outcome), we will show in Section \ref{S:p=s 
on surface} that when $X$ is a surface stackification is in fact unnecessary.

\section{Moduli of stable $\PGL_n$-torsors on surfaces}
\label{S:surfaces}

For the rest of this paper, we assume that $S=\spec k$ with $k$ a
separably closed field and $X/S$ a smooth projective surface with
a fixed ample divisor $H$.

\subsection{Stability of torsors}
\label{S:stability}

We first recall a basic definition.

\begin{defn}
  Given a torsion free sheaf $\ms F$, the \emph{slope\/} of $\ms F$ is
  $\deg\ms F/\rk\ms F$.
\end{defn}

To define stability for $\PGL_n$-torsors, we use the adjoint sheaf.
As described in Section \ref{S:azumaya}, this adjoint sheaf naturally comes with
an algebra structure, which we will use in our definition.

\begin{defn}\label{D:stab}
  An Azumaya algebra $\ms A$ on $X$ is \emph{stable\/} if for all
  non-zero right ideals $\ms I\subset\ms A$ of rank strictly smaller than
  $\rk\ms A$ we have $\mu(\ms I)<0$.
\end{defn}

\begin{remark}
  It is equivalent to quantify over left ideals.  Thus, one could
  state the definition by omitting the word ``right'' and quantifying
  over arbitrary ideals, understood as right or left ideals.  It is of
  course not sufficient to quantify over two-sided ideals.
\end{remark}

\begin{remark}
  This definition is meant to apply only to the classical notion of
  slope-stability, and not to the more refined notion due to
  Gieseker. While such notions of stability using normalized Hilbert
  polynomials are essential for the development of moduli theory using
  Geometric Invariant Theory (GIT), they are somewhat artificial in
  the sense that they no longer correspond to the existence of a
  Hermite-Einstein connection.  (However, recent work of Wang
  \cite{wang1,wang2} has clarified the analytic meaning of
  Gieseker-stability in terms of the existence of certain weak
  Hermite-Einstein connections.)

  One way to understand the compactifications of the stack of
  slope-stable bundles -- using slope-semistability or
  Gieseker-semistability, GIT-bound or purely stacky, etc. -- is that
  each really only serves to impose the kind of inductive strucure on
  the moduli problem necessary to prove theorems about the actual part
  of interest: the open sublocus of slope-stable bundles.  Working in
  a GIT-free manner (which is necessary in the context of twisted
  sheaves) frees us to ignore the subtleties (both algebraic and
  analytic) of Gieseker-stability.  This is pursued in
  \cite{more-moduli}, where the asymptotic properties of moduli are
  proved entirely without GIT.
\end{remark}

\begin{para}
  To relate Definition \ref{D:stab} to the cover of $f_\ast(\B\PGL_n)$
  by the stack of twisted sheaves, we recall some rudiments from the
  theory of Chern classes for twisted sheaves.  A different
  development of the theory of Chern classes for twisted sheaves and
  the relationship to the theory described here is given in \cite{heinloth}.

  Given a coherent $\ms X$-twisted sheaf $\ms F$, we can use the
  rational Chow theory of $\ms X$ to define Chern classes $c_i(\ms
  F)$, $i=1,2$.  (The first Chern class $c_1(\ms F)$ is just the class
  in Chow theory associated to the invertible sheaf $\det\ms F$.)
  There is also a degree map from $d:A_0(\ms X)\to\Q$; this has the
  property that the $0$-cycle supported over a closed point of $X$ has
  degree $1/n$.  We define a normalized degree function by $\deg=nd$.
  Using this degree, we have the following definition.
\end{para}

\begin{defn}
  A torsion free $\ms X$ twisted sheaf $\ms V$ is \emph{stable\/} if
  for every subsheaf $\ms W\subset\ms V$ we have
$$\mu(\ms W)<\mu(\ms V).$$
\end{defn}

As a special case of 2.2.7.22 of \cite{twisted-moduli}, it follows
that if $\ms F$ is a flat family of coherent $\ms X$-twisted sheaves
parametrized by $T$ with trivialized determinant, then the function
$t\mapsto\deg c_2(\ms F_t)$ is locally constant on $T$.  Moreover, by
Proposition 4.3.1.2 of \cite{more-moduli}, we have that stability is
an open condition in a flat family of torsionfree $\ms X$-twisted
sheaves.

\begin{notn}
  Let $\Tw_{\ms X/k}(n,\ms O,c)\subset\Tw_{\ms X/k}(n,\ms O)$ denote
  the open and closed substack parametrizing families such that $\deg
  c_2(\ms F_t)=c$ in each geometric fiber.  Let $\Tw^s_{\ms X/k}(n,\ms
  O,c)\subset\Tw_{\ms X/k}(n,\ms O,c)$ denote the open substack whose
  objects over $T$ are families $\ms F$ such that the fiber $\ms F_t$
  is stable for each geometric point $t\to T$.
\end{notn}

\begin{lem}
  A locally free $\ms X$-twisted sheaf $\ms V$ is stable if and only
  if the Azumaya algebra $\pi_\ast\send(\ms V)$ is stable.
\end{lem}
\begin{proof}
  Given a subsheaf $\ms W\subset\ms V$, a straightforward computation
  shows that $\mu(\shom(\ms V,\ms W))=\mu(\ms W)-\mu(\ms V)$.  On the
  other hand, any right ideal of $\ms A$ has the form $\shom(\ms V,\ms
  W)$ for a subsheaf $\ms W\subset\ms V$.  The result follows.
\end{proof}

\begin{para}
Suppose $B\to S$ is a $k$-scheme and $T\to X_B$ is a $\PGL_n$-torsor.
Let $\ms A$ be the locally free sheaf (of rank $n^2$) associated to
the adjoint torsor (which is a $\GL_{n^2}$-torsor).  Using the
Riemann-Roch theorem, the invariance of Euler characteristic in a flat
family, and the fact that $\det\ms A\cong\ms O$, we see
that the function $b\mapsto\deg c_2(\ms A_b)$ is locally constant on
$B$.  This provides a numerical invariant of a $\PGL_n$-torsor which
is constant in a family.  Given a $\m_n$-gerbe $\ms X\to X$ and an
integer $c$, let $(f_{\ast}^{\ms X}\B\PGL_n)(c)$ be the substack of
$f_{\ast}^{\ms X}\B\PGL_n$ parametrizing families where the locally
free sheaf associated to the adjoint bundle has $\deg c_2=c$ in every
fiber.

Thus, there is a decomposition
$$f_{\ast}\B\PGL_n=\sqcup_{\ms X}\sqcup_c(f^{\ms
  X}_{\ast}\B\PGL_n)(c),$$ where the first disjoint union is taken
over a set of $\m_n$-gerbe representatives for $\H^2(X,\m_n)$ and the
second is taken over $\Z$.  Similarly, there is a decomposition
$$f_\ast(\B\PGL_n)^s=\sqcup_{\ms X}\sqcup_c(f^{\ms X}_\ast(\B\PGL_n)^s$$
of stable loci.
\end{para}

\begin{lem}\label{sec:stability-torsors}
  Given an integer $c$, the closed and open substack $\Tw_{\ms
    X/S}(n,\ms O,c/2n)$ is equal to the preimage of its image in $\ms
  M^{\ms X}_n$.  Similarly, $\Tw^s_{\ms X/k}(n,\ms O,c/2n)$ is equal
  to the preimage of its image in $\ms M^{\ms X}_n$.
\end{lem}
\begin{proof}
  Given a point $p$ of $\ms M^{\ms X}_n$ which lifts into $[\ms
  F]\in\Tw_{\ms X/S}(n,\ms O,c/2n)$, it is easy to see that the full
  preimage of $p$ in $\Tw_{\ms X/S}(n,\ms O)$ is given by the twists
  $\ms F\tensor\ms L$ with $\ms L\in\sPic_{X/S}[n]$.  But these have
  the same (rational) Chern classes as $\ms F$, as $\ms L$ is trivial,
  so they also lie in $\Tw_{\ms X/S}(n,\ms O,c/2n)$.  The second
  statement follows from the fact that $\ms F$ is stable if and only
  if $\ms F\tensor\ms L$ is stable for an invertible sheaf $\ms L$.
\end{proof}

\begin{cor}
  There is an open substack $f_\ast(\B\PGL_n)^s\subset f_\ast(\B\PGL_n)$
  parametrizing families $P\to X_T$ of $\PGL_n$-torsors such that for
  all geometric points $t\to T$ the fiber $P_t\to X_t$ is a stable
  $\PGL_n$-torsor.
\end{cor}
\begin{proof}
  This follows from Lemma \ref{sec:stability-torsors} and the fact
  that $\Tw_{\ms X/k}(n,\ms O)^{\lf}\to f_{\ast}(\B\PGL_n)$ is
  universally submersive.
\end{proof}

Since $\Tw_{\ms X/S}(n,\ms O,c/2n)$ is open (and closed) in $\Tw_{\ms
  X/S}(n,\ms O)$, it follows that there is a well-defined open (and
closed) substack $\ms M^{\ms X}_n(c)$ whose preimage is $\Tw_{\ms
  X/S}(n,\ms O,c/2n)$.  There is an open substack $\ms M^{\ms
  X}_n(c)^s$ whose preimage is $\Tw^s_{\ms X/k}(n,\ms O,c/2n)$.  We
have that each $\ms M^{\ms X}_n(c)$ is quasi-proper and that there is
an open immersion $f^{\ms X}_{\ast}\B\PGL_n(c)\inj\ms M^{\ms X}_n(c)$
and an open immersion $f^{\ms X}_\ast\B\PGL_n(c)^s\inj\ms M^{\ms
  X}_n(c)^s$.  Moreover, each $\ms M^{\ms X}_n(c)$ is covered by
$\Tw_{\ms X/S}(n,\ms O,c/2n)$ in such a way that the fibers are
locally $\m_n$-gerbes over $\Pic_{X/S}[n]$-torsors, and similarly for
the open substacks parametrizing stable objects.  This covering
restricts to a covering of $(f^{\ms X}_{\ast}\B\PGL_n)(c)^s$ by
$\Tw^s_{\ms X/S}(n,\ms O,c/2n)^{\lf}$.

In particular, we have that $\ms M^{\ms X}_n(c)^s$ is irreducible,
separated, \ldots (resp.\ has any local property stable for the
\'etale topology) if (resp.\ if and only if) the same is true for
$\Tw^s_{\ms X/S}(n,\ms O,c/2n)$.  Moreover, there is a virtual
fundamental class for $\ms M^{\ms X}_n(c)^s$ if and only if there is one
for $\Tw^s_{\ms X/S}(n,\ms O,c/2n)$.

\begin{notn}
  Let $\GAz_{\ms X/S}(n)^s$ denote the open substack parametrizing stable
  generalized Azumaya algebras via the isomorphism $\GAz_{\ms
    X/S}(n)\simto(\ms M^{\ms X}_n)^s$ of Section \ref{sec:comp}.
\end{notn}

\subsection{Structure of moduli of twisted sheaves}

The following results show that infinitely many of the spaces
$\Tw^s_{\ms X/k}(n,\ms O,\gamma)$ are non-empty.  It is a geometric
restatement of the fundamental result of de Jong \cite{dejong-per-ind}
on the period-index problem for Brauer classes over functionf ields of
algebraic surfaces.

\begin{lem}\label{L:dejong}
  There is a stable locally free $\ms X$-twisted sheaf of rank $n$.
\end{lem}
A proof of this result may be found in Theorem 4.2.2.3 of
\cite{period-index-paper} and Proposition 5.1.2 of \cite{more-moduli}.

\begin{lem}\label{L:induct}
  Suppose $\ms F$ is a coherent $\ms X$-twisted sheaf of rank $n$ with
  $\deg c_2(\ms F)=\gamma$ and $\det\ms F\cong\ms O$.  For each
  integer $\ell\geq 0$, there is a (noncanonical) subsheaf $\ms
  F_{\ell}\subset\ms F$ such that $\dim\ms F/\ms F_{\ell}=0$, 
  $\det\ms F_{\ell}\cong\ms O$, and $\deg c_2(\ms
  F_{\ell})=\gamma+\ell$.  If $\ms F$ is stable then so is $\ms F_\ell$.
\end{lem}
\begin{proof}
  By induction, it suffices to construct $\ms F_1$.  Choose a point
  $x\in X(k)$ around which $\ms F$ is locally free and let $\ms
  F\tensor\kappa(x)\surj\ms Q$ be a quotient with geometric fiber of
  dimension $1$.  (In other words, given an algebraically closed
  extension field $L/\kappa(x)$ and a map $\spec L\to\ms
  X\tensor\kappa(x)$, the pullback of $\ms Q$ to $\spec L$ is the
  sheaf associated to a one-dimensional vector space.)  We claim that
  $\deg(c_2(\ms Q))=-1$, from which the result follows
  by the multiplicativity of the total Chern polynomial.
  A proof of the claim uses the Grothendieck-Hirzebruch-Riemann-Roch
  theorem for representable morphisms of Deligne-Mumford stacks and
  can be found in the proof of Lemma 3.2.4.8 of \cite{twisted-moduli}
  (where there is an unfortunate sign error in the statement, even
  though the proof is correct!). 

  To deduce stability of $\ms F_1$ from stability of $\ms F$, first
  note that the two sheaves agree in codimension $1$.  Since stability
  depends on a calculation of degree and this calculation depends only
  on a sheaf in codimension $1$, we see that we need only quantify
  over saturated subsheaves (see Definition 1.1.5 of \cite{h-l}),
  which are determined by their values in codimension $1$.  Thus, the
  criterion determining stability of $\ms F$ and $\ms F_1$ quantifies
  over the same set of subsheaves with the same numerical calculations.
\end{proof}

\begin{cor}\label{C:non-empty}
  If $\Tw^s_{\ms X/k}(n,\ms O,\gamma)$ is non-empty then so is $\Tw^s_{\ms
    X/k}(n,\ms O,\gamma+\ell)$ for all integers $\ell\geq 0$.
\end{cor}

The fundamental structure theorem concerning these moduli spaces is
the following.

\begin{thm}\label{T:asymp}
  There exists a constant $C$ such that for all $\gamma\geq C$, 
\begin{enumerate}
\item the open substack $\Tw^s_{\ms X/k}(n,\ms
  O,\gamma)^{\lf}\subset\Tw^s_{\ms X/k}(n,\ms O,\gamma)$ is
  schematically dense;
\item $\Tw^s_{\ms
    X/k}(n,\ms O,\gamma)$ is an irreducible proper normal lci tame
  Deligne-Mumford stack over $k$ whenever it is non-empty;
\item it is non-empty for infinitely many $\gamma$.
\end{enumerate}
\end{thm}
\begin{proof}
The third statement follows immediately from Corollary \ref{C:non-empty}.  For
the proof of the first and second, the reader is referred to paragraph 3.2.4.1
(and especially Theorem 3.2.4.11) of \cite{twisted-moduli}.
\end{proof}

\begin{para}\label{coarse}
Since every object of $\Tw^s_{\ms X/k}(n,\ms O,\gamma)$ is
geometrically stable, it is simple (see, e.g., Corollary 1.2.8 and
Theorem 1.6.6 of \cite{h-l}), i.e., its automorphisms are simply given
by multiplication by scalars (in $\m_n$, since the determinant is
trivialized).  It follows that $\Tw^s_{\ms X/k}(n,\ms O,\gamma)$ is a
$\m_n$-gerbe over its coarse moduli space $\mTw^s_{\ms X/k}(n,\ms
O,\gamma)$.  
\end{para}

\subsection{Consequences for $\ms M^{\ms X}_n$ and $f^{\ms X}_\ast\B\PGL_n$}
\label{S:conseq}

\begin{thm}\label{T:main}
  There is a constant $D$ such that for all $c\geq D$, 
\begin{enumerate}
\item the open substack $f_{\ast}^{\ms X}\B\PGL_n(c)^s$ is
  schematically dense in $\ms M^{\ms X}_n(c)^s$;
\item $\ms M^{\ms
    X}_n(c)^s$ is an irreducible proper normal lci tame Deligne-Mumford
  stack over $k$ whenever it is non-empty;
\item it is non-empty
  for infinitely many $c$.
\end{enumerate}
In particular, the open substack $(f_{\ast}^{\ms X}\B\PGL_n)(c)^s$ is
irreducible (and non-empty for infinitely many $c$).
\end{thm}
\begin{proof}
  The proof follows immediately by combining the covering described at
  the end of Section \ref{sec:compac-by-rig} with Lemma \ref{L:dejong},
  Lemma \ref{L:induct}, and Theorem \ref{T:asymp}.
\end{proof}
Recall that ``lciq singularities'' are by definition
finite quotients of lci singularities.

\begin{cor}\label{C:lciq}
  For sufficiently large $c$, the coarse moduli space $(\ms M^{\ms
    X}_n(c)^s)^{\text{\rm mod}}$ is an irreducible proper normal 
  algebraic space with lciq singularities.
\end{cor}
\begin{proof}
  Lemma \ref{L:pushforwardcovering} gives rise to a finite morphism
  $\mTw^s_{\ms X/k}(n,\ms O,c/2n)\to(\ms M^{\ms
    X}_n(c)^s)^{\textrm{mod}}$ from the coarse space of Paragraph \ref{coarse}
  which is invariant for the natural action of $\Pic_{X/k}[n]$ on
  $\mTw^s_{\ms X/k}(n,\ms O,c/2n)$ and such that the natural map
  $\chi:\mTw^s_{\ms X/k}(n,\ms O,c/2n)/\Pic_{X/k}[n]\to(\ms M^{\ms
    X}_n(c)^s)^{\textrm{mod}}$ is birational.  Since the coarse space of
  a normal tame Deligne-Mumford stack is normal, it follows from
  Zariski's Main Theorem that $\chi$ is an isomorphism.  Since
  $\Tw^s_{\ms X/k}(n,\ms O,c/2n)$ is lci, it follows that $(\ms M^{\ms
    X}_n(c)^s)^{\textrm{mod}}$ is lciq.
\end{proof}

\subsection{Stackification is unnecessary on a surface}\label{S:p=s 
on surface}

Let $f:X\to S$ be a smooth projective relative surface.  We will prove
here that pre-generalized Azumaya algebras on $X$ as in Section \ref{S:gen
  azumaya construction} form a stack on $S$.

Given a pre-generalized Azumaya algebra $\ms A$ on $X$,
Lemma \ref{L:embedding lemma} produces a $\G_{m}$-gerbe $\ms X$, an $\ms
X$-twisted sheaf $\ms F$, and an isomorphism of \emph{generalized\/}
Azumaya algebras $\ms B:=\R\pi_{\ast}\rsend(\ms F)\simto\ms A$.  We
will show that in fact $\ms B$ and $\ms A$ are isomorphic as
\emph{pre-generalized Azumaya algebras\/}.  We will temporarily call
$\ms B$ the \emph{associated twisted derived endomorphism algebra\/}
(or TDEA for short).

\begin{prop}\label{P:genaz is TDEA} 
  Suppose $f:X\to S$ is a smooth (possibly non-proper) relative
  surface over an affine scheme.  Any pre-generalized Azumaya algebra
  $\ms A$ is isomorphic to the associated TDEA in $\PR$.  Furthermore,
  the isomorphisms of two such weak algebras form a sheaf on $S$.
\end{prop}
\begin{proof} 
  By standard arguments (e.g.\ Theorem 3.2.4 of \cite{BBD} or Theorem
  2.1.9 of \cite{abramovich-polishchuk}), it suffices to prove that
  $\ext^{-i}(\ms A,\ms A)=0$ for all $i>0$ (as long as we allow
  $f:X\to S$ to be arbirary with the stated hypotheses).  From the
  definition of pre-generalized Azumaya algebra, we know that $\ms A$
  has cohomology only in degrees $0$ and $1$, that $\ms H^0(\ms A)$
  has totally pure fibers over $S$, and that $\ms H^1(\ms A)$ has
  support with relative dimension $0$.  The natural triangle
$$\ms H^0(\ms A)\to\ms A\to\ms H^1(\ms A)[-1]\xto{+}$$
gives rise to an exact sequence
$$\ext^{-i}(\ms H^1(\ms A),\ms A[1])\to\ext^{-i}(\ms A,\ms A)\to\ext^{-i}(\ms
H^0(\ms A),\ms A).$$
The left-hand group fits into an exact sequence
$$\ext^{-i}(\ms H^1(\ms A),\ms H^0(\ms A))\to\ext^{-i}(\ms H^1(\ms
A),\ms A[1])\to\ext^{-i}(\ms H^1(\ms A),\ms H^1(\ms A))$$
and the right-hand group fits into an exact sequence
$$\ext^{-i}(\ms H^0(\ms A),\ms H^0(\ms A))\to\ext^{-i}(\ms H^0(\ms
A),\ms A)\to\ext^{-i-1}(\ms H^0(\ms A),\ms H^1(\ms A)).$$ (This is
simply an explicit description of a certain spectral sequence, which
is especially simple because $\ms A$ has so few cohomology sheaves.)
We wish to show that the ends of the last two sequences vanish, for
which it is enough to show (using the local-to-global $\ext$-spectral
sequence) that the $\ext$-sheaves $\sext^{\ast}(\ms H^{\ast}(\ms
A),\ms H^{\ast}(\ms A))$ vanish for appropriate indices.  But there
are no negative $\ext$-groups for modules over a ring.  This completes
the proof.
\end{proof}

\begin{remark}
  When $X/S$ is quasi-projective, one can also give an explicit proof
  of Proposition \ref{P:genaz is TDEA} (which does not rely on \cite{BBD}) using
  resolutions by sums of powers of $\ms O(1)$.
\end{remark}

\subsection{Deformation theory and the virtual fundamental class}
\label{S:dtvfc}

Let $k$ be an algebraically closed field and $X/k$ a smooth projective
surface over $k$.  Fix a $\m_n$-gerbe $\ms X\to X$ with $n$ invertible in $k$.

\subsubsection{Perfect obstruction theory for twisted sheaves}
\label{sec:perf-obstr-theory}

Let $\ms F$ be the universal $\ms X$-twisted sheaf on $\ms
X\times\Tw^s_{\ms X/k}(n,\ms O)$.  Write $p$ (resp.\ $q$) for the
projection of $\ms X\times\Tw^s_{\ms X/k}(n,\ms O)$ to $\ms X$ (resp.\
$\Tw^s_{\ms X/k}(n,\ms O)$). Recall that there is a natural
isomorphism $$\LL_{\ms X\times\Tw^s_{\ms X/k}(n,\ms O)}\cong\L
p^{\ast}\LL_{\ms X}\oplus\L q^{\ast}\LL_{\Tw^s_{\ms X/k}(n,\ms O)},$$
and that there is an isomorphism of functors (coming from Grothendieck
duality for $q$) $$\L q^{\ast}\cong\L q^{!}\ltensor\L
p^{\ast}\omega_X^{\vee}[-2].$$ The Atiyah class $$\ms F\to\LL_{\ms
  X\times\Tw^s_{\ms X/k}(n,\ms O)}\ltensor\ms F[1]$$ yields by
projection a map $$\ms F\to\L q^{\ast}\LL_{\Tw^s_{\ms X/k}(n,\ms
  O)}[1]\ltensor\ms F.$$ Since $\ms F$ is perfect, this is equivalent
(by the cher \`a Cartan isomorphism) to a map $$\rsend(\ms F)\to\L
q^{\ast}\LL_{\Tw^s_{\ms X/k}(n,\ms O)}[1],$$ yielding a
map $$\rsend(\ms F)\to\L q^{!}\LL_{\Tw^s_{\ms X/k}(n,\ms
  O)}[-1]\ltensor\L p^{\ast}\omega_X^{\vee}.$$ Applying Grothendieck
duality yields a morphism $$\mf b:\R q_{\ast}\rshom(\ms F,\L
p^{\ast}\omega_X\ltensor\ms F)\to\LL_{\Tw^s_{\ms X/k}(n,\ms O)}[-1]$$ and
restriction to the traceless part finally yields a morphism
$$\mf b_0:\R q_{\ast}\rshom(\ms F,\L
p^{\ast}\omega_X\ltensor\ms F)_0\to\LL_{\Tw^s_{\ms X/k}(n,\ms O)}[-1]$$

\begin{prop}\label{P:perf-obs-tw-sh}
  The shifted map $\mf b_0[1]$ gives a perfect obstruction theory for $\Tw^s_{\ms
    X/k}(n,\ms O)$.
\end{prop}
\begin{proof}
  There are three things to check: that the complex $\R
  q_{\ast}\rshom(\ms F,\L p^{\ast}\omega_X\ltensor\ms F)_0$ is perfect
  of amplitude $[0,1]$, that $\mf b_0$ induces an isomorphism on
  $\H^1$ sheaves, and that $\mf b_0$ induces a surjection on $\H^{0}$
  sheaves.  The first assertion follows from the compatibility of the
  formation of the complex with base change on $\Tw^s_{\ms X/k}(n,\ms
  O)$ (which shows that it is perfect, as its fibers are bounded
  complexes on regular schemes) and the fact that the fibers of $\R^2
  q_{\ast}\rshom(\ms F,\L p^{\ast}\omega_X\ltensor\ms F)_0$ over
  geometric points of $\Tw^s_{\ms X/k}(n,\ms O)$ compute (by Serre
  duality) the traceless endomorphisms of stable sheaves, which must
  be trivial (so that the amplitude is as claimed).  The other two
  assertions will follow from Illusie's theory.  We already know
  (thanks to Illusie) that deformations and obstructions are governed
  by Atiyah classes; we will describe how this allows us to show that
  $\mf b_0$ gives a perfect obstruction theory.

  Since both the domain and codomain of $\mf b_0$ have no cohomology
  above degree $1$, to show that $\mf b_0$ induces an isomorphism on
  $\H^1$ sheaves, it suffices to show this after base change to
  geometric points of $\Tw^s_{\ms X/k}(n,\ms O)$.  Given a geometric
  point $x\to\Tw^s_{\ms X/k}(n,\ms O)$ corresponding to a twisted
  sheaf on $\ms X\tensor\kappa(x)$, we know that $\hom(\LL_{\Tw^s_{\ms
      X/k}(n,\ms O)},\kappa(x))$ is naturally identified with the
  space $T$ of first order determinant-preserving deformations of $F$
  over $\kappa(x)[\eps]$.  Moreover, by the functoriality of our
  construction with respect to base change on $\Tw^s_{\ms X/k}(n,\ms
  O)$, the map $\mf b_0$ is the Serre dual of the Kodaira-Spencer map
  $T\to\ext^1(F,F)_0$ (see e.g.\ Example 10.1.9 of \cite{h-l}).  This is
  well-known to give an isomorphism (e.g., Section 10.2 of \cite{h-l}).

  To show that $\H^0(\mf b_0)$ is surjective, we may proceed as
  follows.  Recall (IV.3.1.8 of \cite{illusie}) that if $A\to A_0$ is
  a small extension of $B$-algebras with kernel $I$ in a topos and $M$
  is an $A_0$-module, then one can find the obstruction to deforming
  $M$ to an $A$-module as the composition
$$M\to\LL_{A_0/B}\ltensor M[1]\to I\ltensor M[2]\to I\tensor M[2],$$
where the first map is the Atiyah class of $M$ with respect to
$A_0/B$, the second map comes from the morphism $\LL_{A_0/B}\to I$
parametrizing the class of the extension $A\to A_0$, and the third map
is the natural augmentation onto the $0$th cohomology module.  To
apply this to our
case, consider a situation 
$$\spec B\leftarrow \spec B_0\to\Tw^s_{\ms X/k}(n,\ms O)$$ with
$B\to B_0$ a small extension of strictly Henselian local rings and
kernel annihilated by the maximal ideal of $B$.  We let $A$ be the
structure sheaf of $\ms X\times\spec B$ and $A_0$ that of $\ms
X\times\spec B_0=\ms X\times\spec B\times_{\spec B}\spec B_0$.  Thus,
we have that $$\LL_{A_0/B}=\L p^{\ast}\LL_{\ms X}\oplus\L
q^{\ast}\LL_{B_0/B}.$$ Moreover, it is clear that the morphism
$\LL_{A_0/B}\to \ms O_{\ms X}\tensor I[1]$ parametrizing the extension
$A\to A_0$ is given by the map $$\LL_{A_0/B}\to\L
q^{\ast}\LL_{B_0/B}\to\L q^{\ast}I[1].$$ By functoriality, composing
the Atiyah class with the natural map $$\LL_{\Tw^s_{\ms X/k}(n,\ms
  O)}|_{B_0}\to \LL_{B_0/B}$$ gives rise to the map
$$\ms F\to\L q^{\ast}\LL_{B_0/B}\ltensor\ms F[1]$$
associated to the projection of the Atiyah class.  Thus, we find that
the obstruction to deforming $\ms F$ over $B$ is the element
corresponding by Serre duality to the composition
$$\R q_{\ast}\rshom(\ms F,\L q^{\ast}\omega\ltensor\ms
F)|_{B_0}\to\LL_{\Tw_{\ms X/k}(n,\ms
  O)}[-1]|_{B_0}\to\LL_{B_0/B}[-1]\to\kappa.$$ On the other hand, the
last two arrows give precisely the obstruction to extending the map
$\spec B_0\to\Tw_{\ms X/k}(n,\ms O)$ to a map $\spec B\to\Tw_{\ms
  X/k}(n,\ms O)$. Thus, the entire composition is trivial if and only
if the composition of the last two maps is trivial.  Since any map
$\LL_{\Tw_{\ms X/k}(n,\ms O)}\to\kappa[1]$ factors through some
deformation situation $B\to B_0$, this shows that $\H^0(\mf b)$ is
surjective.  Using the fact that $n$ is invertible in $k$, and thus
the existence of a splitting trace map, it is easy to see that the
canonical obstruction given here actually lies in the traceless part
of $\ext^2(\ms F,\ms F)$; this shows that in fact $\H^0(\mf b_0)$ is
surjective, as desired.
\end{proof}

The deformation theory described here has a concrete form: given a
generalized Azumaya algebra $\rsend(\ms F)$ on $X$, the first-order
infinitesimal deformations form a pseudo-torsor under the hypercohomology
$\HH^1(X,\rsend(\ms F)_0)$, while there is naturally a class in
$\HH^2(X,\rsend(\ms F)_0)$ giving the obstruction to deforming $\ms
F$.  When $\ms F$ is locally free, so that $\rsend(\ms F)\cong A$ is
an Azumaya algebra, we recover the well-known fact that $\H^1(X,A_0)$
parametrizes deformations of $A$, while $\H^2(X,A_0)$ receives
obstructions.  If $A\cong\shom(\ms V)$ is the sheaf of endomorphisms
of a locally free sheaf on $X$ with trivial(ized) determinant, our
general machine simply says that the deformation and obstruction
theory of the algebra $A$ is the same as the deformation and
obstruction theory of $\ms V$ as a locally free sheaf with trivialized
determinant. (This similarly describes the deformation theory of a
twisted sheaf with trivialized determinant.)

\subsubsection{The virtual fundamental class of $\GAz_{\ms X/k}(n)^s$}
\label{sec:virt-fund-class}

By its construction, $\GAz_{\ms X/k}(n)^s$ is a subfibered category of
the fibered category of weak algebras on $X_{\retale}$.  As such, there is a
universal generalized Azumaya algebra $\ms A$ on $X\times\GAz_{\ms
  X/k}(n)^s$ whose fibers over the moduli space have cohomology class $[\ms
X]$.  If $\pi:\ms Y\to X\times\GAz_{\ms X/k}(n)^s$ is the gerbe of
trivialized trivializations of $\ms A$, then we have that $\ms
A\cong\R\pi_{\ast}\rsend(\ms F)$ for some $\ms Y$-twisted sheaf $\ms
F$.  Moreover, the covering $\Tw^s_{\ms X/k}(n,\ms O)\to\GAz_{\ms
  X/k}(n)^s$ gives rise to an isomorphism $$\rho:\ms
Y\times_{X\times\GAz_{\ms X/k}(n)^s}X\times\Tw^s_{\ms X/k}(n,\ms
O)\simto\ms X\times\Tw^s_{\ms X/k}$$ and an isomorphism $\ms
G\to\rho^{\ast}\ms F$, where $\ms G$ is the universal twisted sheaf on
$\ms X\times\Tw^s_{\ms X/k}(n,\ms O)$.  There results a natural isomorphism
of weak algebras $\rsend(\ms F)\simto\L\rho^{\ast}\ms A$.

Letting $\ms A_0\subset\ms A$ be the traceless part, there is an
induced isomorphism $\rsend(\ms F)_0\simto\L\rho^{\ast}\ms
A_0$. Applying Proposition \ref{P:virt-fund} and Proposition
\ref{P:perf-obs-tw-sh}, we conclude that there is a perfect
obstruction theory $\ms A_0\to\LL_{\GAz_{\ms X/k}(n)^s}$, giving rise to
a virtual fundamental class on $\GAz_{\ms X/k}(n)^s$.

\subsection{A potential application: numerical invariants of division algebras over function fields}
\label{sec:potent-appl-new}

Suppose that the cohomology class $\alpha$ of $\ms X$ in $\H^2(X,\G_m)$ has
order $n$.  If $\ms A$ is an Azumaya algebra of degree $n$ with
cohomology class $\alpha$, then the generic fiber of $\ms A$ must be a
finite dimensional central division algebra $D$ over the function field
$k(X)$.  In this case, we have an especially nice description of the
stable locus.

\begin{lem}
  When $[\ms X]$ has order $n$ in $\H^2(X,\G_m)$, any $\PGL_n$-torsor
  $T$ with class $\alpha$ is stable.
\end{lem}
\begin{proof}
  Indeed, if $\ms A$ is the Azumaya algebra associated to $T$ then any
  non-zero right ideal must have rank $n^2$, since the generic fiber
  of $\ms A$ is a division algebra.
\end{proof}

Thus, $\GAz_{\ms X/k}(n)^s$ is a proper Deligne-Mumford stack which
carries a virtual fundamental class, as described in Section
\ref{sec:virt-fund-class}.  The following question was asked by de Jong.

\begin{ques}
  Does the virtual class $[\GAz_{\ms X/k}(n)^s]^{\textrm{vir}}$ lead
  to any new numerical invariants attached to $D$?
\end{ques}

Via Proposition \ref{P:virt-fund}(ff), any invariants coming
from $[\GAz_{\ms X/k}(n)^s]^{\textrm{vir}}$ will be closely related to
similar numbers attached to $[\Tw^s_{\ms X/k}(n,\ms
O)^s]^{\textrm{vir}}$.  One might expect the latter invariants to be
related to Donaldson invariants.

One interesting direct comparison might arise as follows: suppose given a
family of surfaces $\mc X\to S$ and a class $\alpha\in\H^2(\mc
X,\m_n)$ with $n$ invertible on the base, such that there are two
geometric points $0,1\to S$ such that $\alpha|_{\mc X_0}$ has order
$n$ in $\H^2(\mc X_0,\G_m)$ and $\alpha|_{\mc X_1}$ vanishes in
$\H^2(\mc X_1,\G_m)$.  (This happens whenever there is jumping in the
rank of the N\'eron-Severi group in the family.) 
Assuming one could prove deformation invariance
of whatever invariants one eventually defines, one would then be able
to give a direct comparision between the division-algebra invariants
attached to $\mc X_0$ and the classical invariants attached to $\mc X_1$.

\end{document}